\documentclass[12pt]{amsart}
\usepackage[utf8]{inputenc}
\usepackage[top=1in, left=1in, right=1in, bottom=1in]{geometry}
\usepackage{color, amsmath,amssymb,mathtools,graphicx, mathrsfs}
\usepackage{epsfig,fancyhdr}
\usepackage{float}
\usepackage{setspace}
\usepackage{epstopdf}
\usepackage{import}
\usepackage{lineno}
\usepackage{stackrel}
\usepackage[allcolors = blue,colorlinks]{hyperref}
\usepackage[abs]{overpic}
\usepackage[center]{caption}
\usepackage{subcaption}
\usepackage{tikz-cd}
\usepackage{enumitem}
\counterwithin{figure}{section}
\counterwithin{table}{section}
\usepackage{contour}
\contourlength{.08em}
\usepackage{cite}
\usepackage{marginnote}
\makeatletter 
\let\c@table\c@figure
\let\c@lstlisting\c@figure
\makeatother

\numberwithin{equation}{section}
\theoremstyle{plain}
\newtheorem{theorem}[equation]{Theorem}
\newtheorem{lemma}[equation]{Lemma}
\newtheorem{proposition}[equation]{Proposition}
\newtheorem{corollary}[equation]{Corollary}

\newtheorem{conjecture}[equation]{Conjecture}

\newtheorem*{namedtheorem}{\theoremname}
\newcommand{\theoremname}{testing}
\newenvironment{named}[1]{\renewcommand{\theoremname}{#1}\begin{namedtheorem}}{\end{namedtheorem}}

\theoremstyle{definition}
\newtheorem{definition}[equation]{Definition}
\newtheorem{example}[equation]{Example}

\long\def\@savemarbox#1#2{\global\setbox#1\vtop{\hsize\marginparwidth 
  \@parboxrestore\tiny\raggedright #2}}
\newcommand{\HH}{{\mathbb{H}}}

\newcommand{\CC}{{\mathbb{C}}}
\newcommand{\QQ}{{\mathbb{Q}}}

\newcommand{\calT}{\mathcal{T}}
\newcommand{\calO}{\mathcal{O}}
\newcommand{\calB}{\mathcal{B}}
\newcommand{\calP}{\mathcal{P}}
\newcommand{\scrA}{\mathscr{A}}

\newcommand{\PSL}{\operatorname{PSL}}
\newcommand{\SL}{\operatorname{SL}}
\newcommand{\bdy}{\partial}
\newcommand{\NZ}{\operatorname{NZ}}
\newcommand{\In}{\operatorname{In}}

\title{On Geometric triangulations of double twist knots}

\author{Dionne Ibarra}
\address{School of Mathematics, Monash University, VIC 3800, Australia.}
\email{{\rm \textcolor{blue}{dionne.ibarra@monash.edu}}}

\author{Daniel V. Mathews}
\address{School of Mathematics, Monash University, VIC 3800, Australia.}
\email{{\rm \textcolor{blue}{daniel.mathews@monash.edu}}}

\author{Jessica S. Purcell}
\address{School of Mathematics, Monash University, VIC 3800, Australia.}
\email{{\rm \textcolor{blue}{jessica.purcell@monash.edu}}}

\begin{document}


\begin{abstract}
In this paper we construct two different explicit triangulations of the family of double twist knots $K(p,q)$ using methods of triangulating Dehn fillings, with layered solid tori and their double covers. One construction yields the canonical triangulation, and one yields a triangulation that we conjecture is minimal. We prove that both are geometric, meaning they are built of positively oriented convex hyperbolic tetrahedra. We use the conjecturally minimal triangulation to present eight equations cutting out the A-polynomial of these knots.
\end{abstract}

\maketitle

\section{Introduction}

Triangulations of 3-manifolds are of fundamental importance in low-dimensional topology: They have contributed to progress in algorithmic topology~\cite{haken, SnapPy, regina}, hyperbolic geometry~\cite{thurstonsnotes, NeumannZagier, SWcanonicaltriangulations}, representation theory~\cite{GGZ:representations}, and quantum topology~\cite{TuraevViro, 3DIndex}. However, there remain many open questions about 3-manifold triangulations. For example, considering only ideal triangulations of hyperbolic 3-manifolds, it is still difficult to determine when a triangulation is \emph{minimal}, i.e.\ using as few tetrahedra as possible; see~\cite{JRST,JRST2}. It is an open conjecture that a cusped hyperbolic 3-manifold has a triangulation that is convex with respect to the hyperbolic metric, that is, \emph{geometric}; see~\cite{Yoshida:TriangConjecture, HPhighlytwisted}. There is a \emph{canonical decomposition} into ideal polyhedra, but it remains difficult to relate its properties to the topology of the manifold, and difficult to determine when this is a triangulation; see~\cite{SWcanonicaltriangulations}. 

This paper considers minimal, geometric, and canonical triangulations.
We consider a simple family of 3-manifolds that has been well-studied: the double twist knots.
General double twist knots $K(p,q)$ are 2-bridge knots that have exactly two twist regions, and we define these to contain $2p$ and $2q$ negative crossings, respectively.  Thus the knot $K(p,q)$ is obtained from the complement of the Borromean rings by performing $-1/p$ Dehn filling on one link component, and $-1/q$ Dehn filling on the other, as in Figure~\ref{fig:BorromeanDehnfilling}.

\begin{figure}
  \centering
  \includegraphics{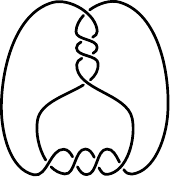}
  \hspace{.3in}
\begingroup%
  \makeatletter%
  \providecommand\color[2][]{%
    \errmessage{(Inkscape) Color is used for the text in Inkscape, but the package 'color.sty' is not loaded}%
    \renewcommand\color[2][]{}%
  }%
  \providecommand\transparent[1]{%
    \errmessage{(Inkscape) Transparency is used (non-zero) for the text in Inkscape, but the package 'transparent.sty' is not loaded}%
    \renewcommand\transparent[1]{}%
  }%
  \providecommand\rotatebox[2]{#2}%
  \newcommand*\fsize{\dimexpr\f@size pt\relax}%
  \newcommand*\lineheight[1]{\fontsize{\fsize}{#1\fsize}\selectfont}%
  \ifx\svgwidth\undefined%
    \setlength{\unitlength}{96.11441517bp}%
    \ifx\svgscale\undefined%
      \relax%
    \else%
      \setlength{\unitlength}{\unitlength * \real{\svgscale}}%
    \fi%
  \else%
    \setlength{\unitlength}{\svgwidth}%
  \fi%
  \global\let\svgwidth\undefined%
  \global\let\svgscale\undefined%
  \makeatother%
  \begin{picture}(1,0.91099835)%
    \lineheight{1}%
    \setlength\tabcolsep{0pt}%
    \put(0,0){\includegraphics[width=\unitlength,page=1]{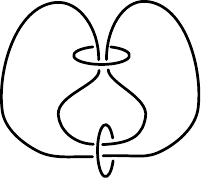}}%
    \put(0.58030736,0.69603803){\color[rgb]{0,0,0}\makebox(0,0)[lt]{\lineheight{1.25}\smash{\begin{tabular}[t]{l}$-1/p$\end{tabular}}}}%
    \put(0.57339326,0.02067645){\color[rgb]{0,0,0}\makebox(0,0)[lt]{\lineheight{1.25}\smash{\begin{tabular}[t]{l}$-1/q$\end{tabular}}}}%
  \end{picture}%
\endgroup%

  \caption{Left: $K(2,2)$. Right: The double twist knot $K(p,q)$ is obtained by Dehn filling the Borromean rings.}
  \label{fig:BorromeanDehnfilling}
\end{figure}

Double twist knots are 2-bridge knots. Therefore they have a well known triangulation, the Sakuma--Weeks triangulation~\cite{SWcanonicaltriangulations}, which is known to be canonical~\cite{GueritaudFuter:Canonical, ASWY}. However, that triangulation is very far from minimal. In this paper, we present a new triangulation of the double twist knots, $\calT_{K(p,q)}$, which uses fewer tetrahedra.
These triangulations are constructed via triangulations of Dehn fillings, using layered solid tori. We prove that they are geometric.

\begin{named}{Theorem~\ref{Thm:geometric}}
For every $p,q\geq 2$, the ideal triangulation $\calT_{K(p,q)}$ of the complement of $K(p,q)$ constructed in Section~\ref{Sec:mintriangulation} is geometric.
\end{named}

In fact, it follows from work of Ham and Purcell~\cite{HPhighlytwisted} that eventually, when both $p$ and $q$ are sufficiently large, triangulations $\calT_{K(p,q)}$ must be geometric. However, that result is not effective, so values of $p$ and $q$ that ensure a geometric triangulation cannot be determined from that paper. Theorem~\ref{Thm:geometric} shows they are, in fact, always geometric. 

We conjecture they are minimal.

\begin{conjecture}
For $p,q\geq 2$, the ideal triangulations $\calT_{K(p,q)}$ of the complement of $K(p,q)$ constructed in Section~\ref{Sec:mintriangulation} realise the triangulation complexity for this 3-manifold. That is, these ideal triangulations give the minimal number of tetrahedra required to triangulate the $K(p,q)$ knot complement. 
\end{conjecture}

Evidence for this conjecture comes from SnapPy: Knots $K(2,2)$, $K(2,3)$, $K(2, 4)$, $K(3,3)$ and $K(3,4)$ all appear in the SnapPy census, with names $s648$, $v2284$, $t05803$, $t07682$, and $o9\_20397$, respectively.
By construction, the number of tetrahedra for a manifold in the census is minimal over all triangulations of that manifold. 
The numbers of tetrahedra in the census agree with the numbers of tetrahedra in our triangulations. For larger $p$ and $q$, there are too many tetrahedra for the triangulation to appear in the census. However, for higher $p,q$ the triangulation is obtained by adjusting a layered solid torus in a Dehn filled manifold. It is conjectured that performing this move on a minimal triangulation yields a new minimal triangulation; see for example Thompson~\cite{Thompson:TriangMagic}. 

A similar conjecture for twist knots appears in~\cite{aribi2022geometric}. If we adjust our triangulation by folding one cusp (as in the construction of Section~\ref{Sec:CanonicalConstruction} for $p=2$), and performing the triangulated Dehn filling on the other as described in the construction of Section \ref{Sec:mintriangulation}, we obtain a triangulation of twist knot complements with the same number of tetrahedra as in that paper.

\subsection{Relation to canonical triangulation}
Our conjectured minimal triangulation is constructed by taking a Dehn filling of a triangulation of the Borromean rings complement. The canonical triangulation can also be seen as a triangulated Dehn filling. We also work through this construction in Section~\ref{Sec:CanonicalConstruction}, from the Dehn filling perspective. This is similar to work done in Ham and Purcell~\cite{HPhighlytwisted}, but here we provide more details, including tetrahedron labels, to make it easier to compute with this triangulation. One goal of writing this paper is to present the triangulations with sufficient detail to ensure ease of applicability for further applications.

\subsection{Application: deformation variety equations}
Finally, we conclude the paper with an application to computing deformation varieties and A-polynomials for double twist knots.  Thurston showed that a hyperbolic 3-manifold with a torus boundary component has a space of deformations through incomplete hyperbolic structures that is 2-real (or 1-complex) dimensional. Such a structure gives rise to a representation of the fundamental group into $\PSL(2,\CC)$, leading to studies of character varieties, or varieties of representations of the fundamental group in to $\SL(2,\CC)$~\cite{CullerShalen}. The two degrees of freedom can be encoded by variables representing curves on the boundary torus. This leads to the definition of the A-polynomial~\cite{CCGLS}, which can be viewed as a parameterisation of the space of incomplete hyperbolic structures of a knot complement. The $\PSL$-A-polynomial can be computed via representations of fundamental groups into $\PSL(2,\CC)$ as in~\cite{BoyerZhang}. A similar picture can be obtained using triangulations rather than group representations. The space of all hyperbolic structures satisfying triangulation gluing equations is called the deformation variety (also called the shape variety or the gluing variety) of the 3-manifold. Gluing and cusp equations give a $\PSL$-A-polynomial that is closely related to the original A-polynomial defined by representation theory; see Champanerkar~\cite{ChampApoly}.

This paper considers deformation varieties and A-polynomials for double twist knots. We focus on these knots because of a relationship with the AJ Conjecture, which conjectures that a quantised version of the A-polynomial annihilates the coloured Jones polynomial~\cite{GAJconjecture}. 
While the AJ Conjecture
has been proved for most 2-bridge knots \cite{LeTr}, it remains open for the family of double twist knots $K(p,p)$. 
The A-polynomials of these knots have been considered from the representation perspective. Petersen found functions cutting out the A-polynomial of $K(p,p)$ in terms of $p$~\cite{PetersenApolynomial2bridge}. In this paper, we find a parametrisation using triangulations, for more general $p$ and $q$. 

\begin{named}{Corollary~\ref{Cor:Apoly}}
For $p, q\geq 3$, eight equations define the $\PSL$-$A$-polynomial of the double twist knot $K(p,q)$, including six Ptolemy equations and two ``Tail'' equations.
\end{named}

The Ptolemy equations produced in Corollary~\ref{Cor:Apoly} are relatively simple, with low degree in the variables. The two Tail equations are more complicated, arising from Dehn fillings, which makes it challenging to reduce the equations further, particularly when $p$ and $q$ are large. However, all equations have the advantage that they are explicit, rather than defined recursively.

\subsection{Acknowledgements}
This research was partially supported by the Australian Research Council, grant DP210103136. Ibarra and Purcell were also supported by ARC DP240102350. The colour pallette was chosen to be accessible. 

\section{Triangulation of the Borromean rings complement}

Throughout this paper we will be considering triangulations of knot and link complements. We begin with the Borromean rings $\calB$, shown on the right of Figure~\ref{fig:BorromeanDehnfilling}. 
Thurston proved that $S^3-\calB$ decomposes into the union of two ideal octahedra~\cite{thurstonsnotes}. We use here the decomposition described, for example in~\cite[Appendix]{Lackenby:AltVolume}; see also~\cite[Chapter~7]{PurcellHyperbolicKnotTheory}.

The diagram of the Borromean rings $\calB$ on the left of Figure~\ref{fig:idealpolyhedronlabelled} consists of two \emph{crossing circles} bounding (shaded) discs meeting the plane of projection transversely at right angles, and a \emph{knot strand} lying within the plane of projection. Slice along the projection plane, cutting the link complement into two identical halves. The shaded discs bounded by crossing circles have been cut in half; slice each shaded half-disc open up the middle like pita bread, and flatten onto the plane of projection. Then collapse the remnants of the link diagram to single ideal points. This gives two ideal octahedra, one of which is shown in Figure~\ref{fig:constructcusp2}.

\begin{figure}[ht]
    \centering
    \begin{subfigure}{.3\textwidth}
    \centering
     $\vcenter{\hbox{\begin{overpic}[scale=.85]{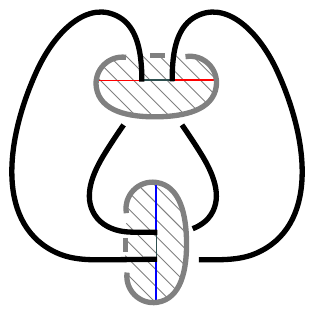}
      \put(67, 28){\contour{white}{$x$}}
     \put(61, 87){\contour{white}{$y$}}
     \put(45, 87){\contour{white}{$r$}}
     \put(78, 87){\contour{white}{$r'$}}
 \put(67, 14){\contour{white}{$p$}}
\end{overpic}}} $
    \caption{Borromean rings with shaded discs.} \label{fig:constructcusp1}
    \end{subfigure}
  \begin{subfigure}{.3\textwidth}
    \centering
     $\vcenter{\hbox{\begin{overpic}[scale=.75]{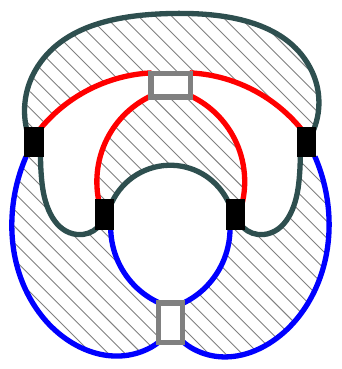}
      \put(100, 40){\contour{white}{$x$}}
       \put(18, 40){\contour{white}{$x$}}
     \put(60, 134){\contour{white}{$y$}}
       \put(60, 64){\contour{white}{$y$}}
     \put(30, 105){\contour{white}{$r$}}
     \put(90, 105){\contour{white}{$r'$}}
      \put(44, 80){\contour{white}{$r$}}
     \put(70, 80){\contour{white}{$r'$}}
 \put(107, 10){\contour{white}{$p$}}
       \put(10, 10){\contour{white}{$p$}}
\end{overpic}}} $
    \caption{Labeled octahedron with shaded faces.} \label{fig:constructcusp2}
    \end{subfigure}
    \begin{subfigure}{.3\textwidth}
    \centering
     $\vcenter{\hbox{\begin{overpic}[scale=.5]{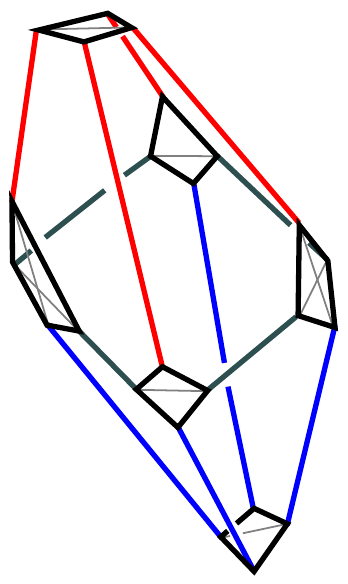}
     \put(26,55){$x$}
     \put(55,85){$x$}
      \put(13,78){$y$}
       \put(64,52){$y$}
       \put(-3, 110){$r$}
       \put(17, 106){$r$}
       \put(60, 110){$r'$}
        \put(29, 115){\contour{white}{$r'$}}
\put(20, 35){$p$}
         \put(60, 35){\contour{white}{$p$}}
\end{overpic}}} $
    \caption{Labelled triangulated truncated octahedron.} \label{fig:constructcusp3}
    \end{subfigure}
    \caption{Constructing the cusp neighbourhood.}
    \label{fig:idealpolyhedronlabelled}
\end{figure}
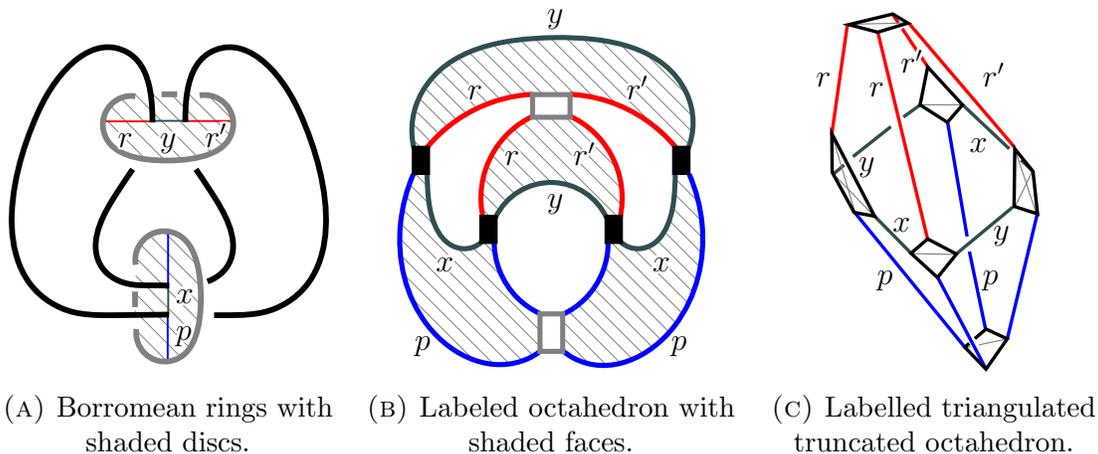

\subsection{Borromean rings: 8 tetrahedra}\label{Sec:8TetB}
A triangulation of the complement of the Borromean rings $S^3-\calB$ is obtained by subdividing each octahedron into four tetrahedra.
The face pairing information of the triangulation is given in Table~\ref{tab:triangulationofborromeanrings}, in the format of Regina~\cite{regina}. 
It was obtained by dropping an edge from one ideal vertex corresponding to a cusp in the knot strand to an opposite ideal vertex, also in the knot strand, and then stellating. This subdivides truncated ideal vertices as shown in Figure~\ref{fig:constructcusp3}. 

\begin{table}
  \centering
  \begin{tabular}{c||c|c|c|c}
    $\Delta$  & 012 & 013 & 023 & 123 \\
    \hline  \hline
    $\Delta_1^1 =$ 0 & 3(013) & 1(012) & 1(013) & 5(132) \\ \hline 
    $\Delta_2^1 =$ 1 & 0(013) & 0(023) & 2(021) & 4(132) \\ \hline
    $\Delta_3^1 =$ 2 & 1(032) & 3(023) & 3(021) & 6(321) \\ \hline
    $\Delta_4^1 =$ 3 & 2(032) & 0(012) & 2(013) & 7(132) \\ \hline
    $\Delta_1^2 =$ 4 & 7(013) & 5(012) & 5(013) & 1(132) \\ \hline
    $\Delta_2^2 =$ 5 & 4(013) & 4(023) & 6(021) & 0(132) \\ \hline
    $\Delta_3^2 =$ 6 & 5(032) & 7(023) & 7(021) & 2(321) \\ \hline
    $\Delta_4^1 =$ 7 & 6(032) & 4(012) & 6(013) & 3(132) 
  \end{tabular}
  \caption{Eight-tetrahedron triangulation of the Borromean rings complement.}
  \label{tab:triangulationofborromeanrings}
\end{table}

Under this triangulation, the three cusps of the Borromean rings inherit the cusp triangulation shown in Figure~\ref{fig:triangulation}. 

\begin{figure}
    \centering
    \begin{subfigure}{.3\textwidth}
    \centering
$\vcenter{\hbox{\begin{overpic}[scale=.65]{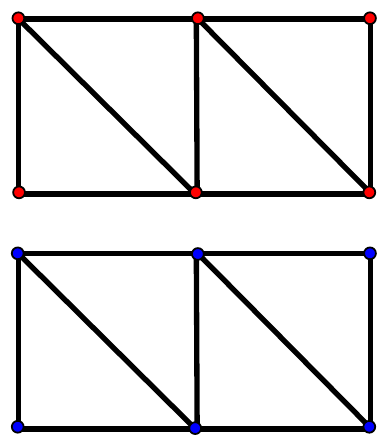}
      \put(3, 138){$r$}
      \put(3, 70){$r$}
       \put(60, 138){\contour{white}{$r'$}}
      \put(60, 70){\contour{white}{$r'$}}
      \put(114, 138){$r$}
      \put(114, 70){$r$}
       \put(15, 95){$\Delta_1^1$}
       \put(35, 115){$\Delta_2^1$}
        \put(70, 95){$\Delta_3^1$}
         \put(90, 115){$\Delta_4^1$}
      \put(15, 20){$\Delta_1^2$}
      \put(35, 40){$\Delta_2^2$}
       \put(70, 20){$\Delta_3^2$}
        \put(90, 40){$\Delta_4^2$}
       \put(64, 10){\contour{white}{$p$}}
      \put(53.5, 52){\contour{white}{$p$}}
\end{overpic}}} $
    \caption{Two cusp neighborhoods to be filled.} \label{fig:cusp1}
    \end{subfigure}
  \begin{subfigure}{.5\textwidth}
    \centering
     $\vcenter{\hbox{\begin{overpic}[scale=.65]{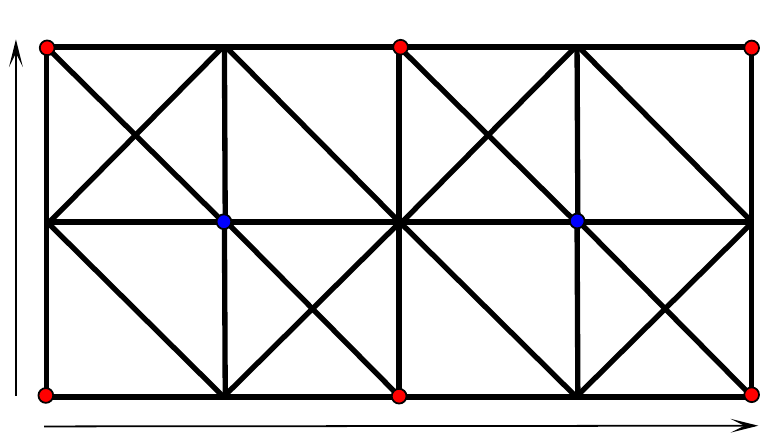}
     \put(12, 127){$r$}
     \put(-4, 125){$\mu$}
      \put(6, 60){$y$}
       \put(67, 127){\contour{white}{$x$}}
      \put(121, 127){$r'$}
       \put(179, 127){$x$}
    \put(231, 125){$r$}
    \put(121, 60){\colorbox{white}{\makebox(.7,3){$y$}}}
    \put(238, 60){$y$}
     \put(67, 60){\colorbox{white}{\makebox(.7,3){$p$}}}
      \put(238, 0){$\lambda$}
       \put(17, 94){$\Delta_1^1$}
        \put(100, 105){$\Delta_1^1$}
        \put(146, 110){$\Delta_1^1$}
        \put(34, 110){$\Delta_2^1$}
         \put(210, 105){$\Delta_2^1$}
         \put(128, 94){$\Delta_2^1$}
          \put(25, 25){$\Delta_3^1$}
        \put(106, 39){$\Delta_3^1$}
         \put(200, 20){$\Delta_3^1$}
          \put(90, 20){$\Delta_4^1$}
           \put(135, 25){$\Delta_4^1$}
           \put(216, 39){$\Delta_4^1$}
            \put(34, 75){$\Delta_2^2$}
            \put(51, 94){$\Delta_1^2$}
             \put(161, 94){$\Delta_2^2$}
    \put(80, 85){$\Delta_2^2$}
    \put(146, 75){$\Delta_1^2$}
    \put(190, 85){$\Delta_1^2$}
    \put(45, 50){$\Delta_3^2$}
    \put(155, 50){$\Delta_4^2$}
    \put(74, 39){$\Delta_4^2$}
    \put(183, 39){$\Delta_3^2$}
     \put(90, 55){$\Delta_3^2$}
      \put(200, 55){$\Delta_4^2$}
\end{overpic}}} $
    \caption{Labeled and triangulated knot strand cusp neighborhood.} \label{fig:cusp2}
    \end{subfigure}
    \caption{Cusp neighbourhoods of the Borromean rings, 8-tetrahedron triangulation.}
    \label{fig:triangulation}
\end{figure}
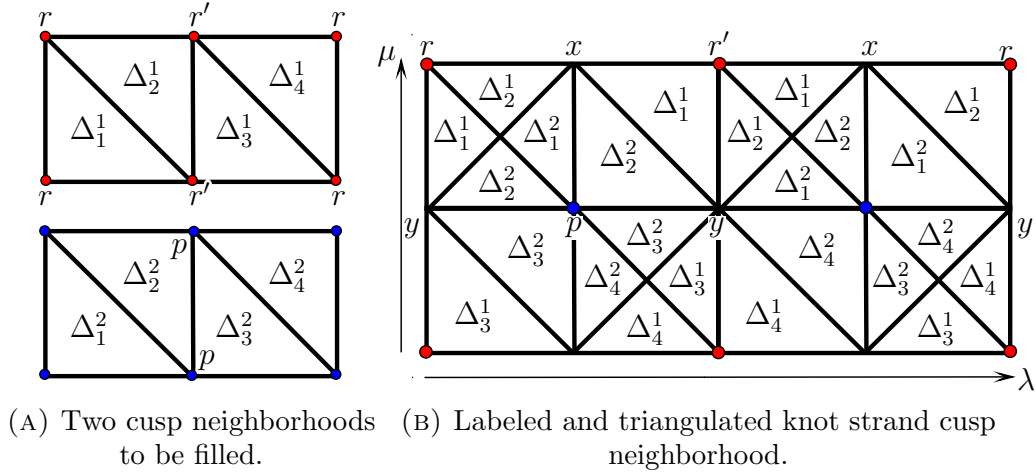

We will eventually wish to Dehn fill, replacing each crossing circle cusp with a solid torus. Removing the cusp can be achieved for one of the crossing circles by removing all four tetrahedra meeting the corresponding cusp. These four tetrahedra glue to form a region $T^2\times [0,\infty)$, with $T^2\times\{0\}$ containing two ideal points, forming a 2-punctured torus. Thus when we remove the tetrahedra, the cusp is removed, leaving only the 2-punctured torus boundary $T^2\times\{0\}$. We will (eventually) glue in a triangulated solid torus with a matching 2-punctured torus boundary; this will be described in Section~\ref{Sec:DehnFillingSolidTori}. 

Meanwhile, if we remove tetrahedra corresponding to both crossing circles of Figure~\ref{fig:BorromeanDehnfilling}, then we will be removing all eight tetrahedra, leaving only the 2-punctured tori boundary. These glue as shown in Figure~\ref{fig:gluingcrossingcirclecusps}. That is, there is a reflection in each square along a diagonal as illustrated in that figure; this is also described in~\cite[Lemma~6.2]{HPhighlytwisted}. We will need this information to attach solid tori correctly. 

\begin{figure}
  \centering
$\vcenter{\hbox{\begin{overpic}[scale=.65]{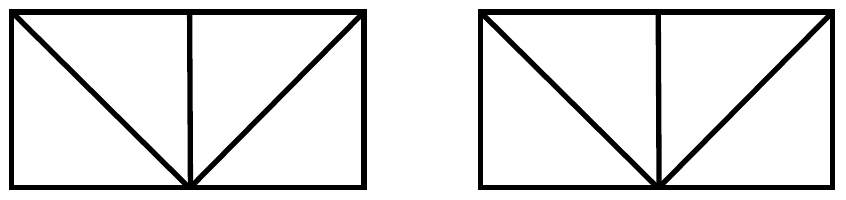}
      \put(25, 30){\contour{white}{$F_1$}}
       \put(80, 30){\contour{white}{$F_2$}}
        \put(175, 25){\rotatebox{90}{\reflectbox{\contour{white}{$F_1$}}}}
        \put(228, 35){\rotatebox{-90}{\reflectbox{\contour{white}{$F_2$}}}}
\end{overpic}}} $
    \caption{Gluing the crossing circle cusps of the Borromean rings.}
    \label{fig:gluingcrossingcirclecusps}
\end{figure}
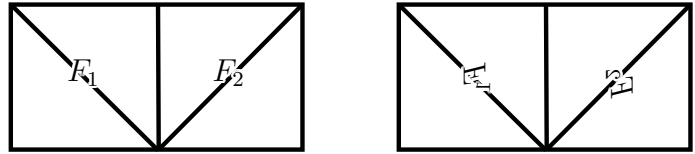 

\subsection{Borromean rings: 10 tetrahedra}\label{Sec:10TetB}
To obtain a construction of the complement of $K(p,q)$ with smaller numbers of tetrahedra for large $p,q$, we start with a different triangulation of the Borromean rings complement $S^3-\calB$.  Each crossing circle cusp from the construction of Section~\ref{Sec:8TetB}
is triangulated by four tetrahedra with boundary forming a 2-punctured torus. Replace these four tetrahedra by five tetrahedra as follows. Remove the four tetrahedra. Attach a \emph{triangulated prism}: this is a thickened triangle, of the form $\mbox{Triangle}\times I$, with two of the faces $\mbox{Edge}\times I$ triangulated to match the faces of the 2-punctured torus, and the third $\mbox{Edge}\times I$ triangulated with an opposite diagonal, as in Figure~\ref{fig:triangulationtwoprism}~(A). This prism can be subdivided into three ideal tetrahedra. The third side $\mbox{Edge}\times I$ of the prism is now facing the cusp, and identifications on its edges glue it into a 1-punctured torus. To obtain the Borromean rings complement, attach to that two tetrahedra running into the cusp, with their bases shown as in Figure~\ref{fig:triangulationtwoprism}~(B). Note this is the triangulation of \cite[Proposition~2.4]{HPhighlytwisted}. 

\begin{figure}
    \centering
  \begin{subfigure}{.25\textwidth}
    \centering
     $\vcenter{\hbox{\begin{overpic}[scale=.65]{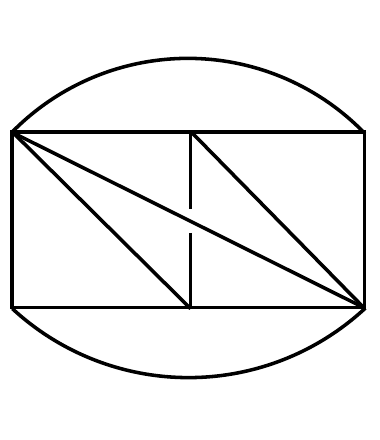}
     \put(15, 60){$\pmb 2$}
      \put(7, 45){\fontsize{7}{7}$3$}
      \put(42, 45){\fontsize{7}{7}$1$}
      \put(7, 80){\fontsize{7}{7}$2$}
     \put(70, 50){$\pmb 0$}
     \put(63, 80){\fontsize{7}{7}$3$}
      \put(89, 45){\fontsize{7}{7}$0$}
         \put(63, 45){\fontsize{7}{7}$2$}
     \put(95, 80){$\pmb 1$}
      \put(70, 89){\fontsize{7}{7}$1$}
      \put(107, 89){\fontsize{7}{7}$3$}
      \put(107, 55){\fontsize{7}{7}$2$}
     \put(40, 80){$\pmb 0$}
      \put(50, 89){\fontsize{7}{7}$3$}
       \put(23, 89){\fontsize{7}{7}$1$}
       \put(52, 52){\fontsize{7}{7}$2$}
     \put(57, 23){$\pmb 2$}
     \put(16, 34){\fontsize{7}{7}$3$}
     \put(99, 34){\fontsize{7}{7}$0$}
      \put(50, 34){\fontsize{7}{7}$1$}
     \put(57, 105){$\pmb 1$}
      \put(16, 99){\fontsize{7}{7}$0$}
      \put(99, 99){\fontsize{7}{7}$3$}
       \put(50, 99){\fontsize{7}{7}$1$}
\end{overpic}}} $
    \caption{Prism 1 replacing tetrahedra from the first crossing circle.} \label{fig:prism1}
    \end{subfigure}
     \begin{subfigure}{.21\textwidth}
    \centering
     $\vcenter{\hbox{\begin{overpic}[scale=.65]{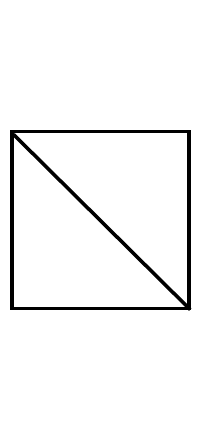}
    \put(20,60){$\pmb 6$}
      \put(8,45){\fontsize{7}{7}$2$}
     \put(42,45){\fontsize{7}{7}$3$}
     \put(8,80){\fontsize{7}{7}$1$}
      \put(33,73){$\pmb 7$}
     \put(51,89){\fontsize{7}{7}$1$}
     \put(51,55){\fontsize{7}{7}$2$}
     \put(14,89){\fontsize{7}{7}$3$}
\end{overpic}}} $
    \caption{Two tetrahedra glued on top of prism 1.} \label{fig:ontopofprism1}
    \end{subfigure}
      \begin{subfigure}{.25\textwidth}
    \centering
     $\vcenter{\hbox{\begin{overpic}[scale=.65]{Figures/prism.pdf}
     \put(15, 60){$\pmb 5$}
     \put(70, 50){$\pmb 3$}
     \put(95, 80){$\pmb 4$}
     \put(40, 80){$\pmb 3$}
     \put(57, 23){$\pmb 5$}
     \put(57, 105){$\pmb 4$}
     \put(7, 45){\fontsize{7}{7}$3$}
     \put(16, 34){\fontsize{7}{7}$3$}
      \put(99, 34){\fontsize{7}{7}$0$}
       \put(50, 34){\fontsize{7}{7}$1$}
     \put(42, 45){\fontsize{7}{7}$1$}
      \put(7, 80){\fontsize{7}{7}$2$}
      \put(16, 99){\fontsize{7}{7}$0$}
      \put(99, 99){\fontsize{7}{7}$3$}
       \put(50, 99){\fontsize{7}{7}$1$}
        \put(63, 80){\fontsize{7}{7}$1$}
        \put(50, 89){\fontsize{7}{7}$1$}
        \put(70, 89){\fontsize{7}{7}$1$}
        \put(23, 89){\fontsize{7}{7}$2$}
         \put(107, 89){\fontsize{7}{7}$3$}
          \put(107, 55){\fontsize{7}{7}$2$}
           \put(89, 45){\fontsize{7}{7}$0$}
           \put(52, 52){\fontsize{7}{7}$3$}
            \put(63, 45){\fontsize{7}{7}$3$}
\end{overpic}}} $
    \caption{Prism 2 replacing tetrahedra from the second crossing circle.} \label{fig:prism2}
    \end{subfigure}
     \begin{subfigure}{.21\textwidth}
    \centering
     $\vcenter{\hbox{\begin{overpic}[scale=.65]{Figures/prismontop.pdf}
    \put(20,60){$\pmb 8$}
     \put(8,45){\fontsize{7}{7}$2$}
     \put(42,45){\fontsize{7}{7}$3$}
     \put(8,80){\fontsize{7}{7}$1$}
      \put(33,73){$\pmb 9$}
     \put(51,89){\fontsize{7}{7}$1$}
     \put(51,55){\fontsize{7}{7}$2$}
     \put(14,89){\fontsize{7}{7}$3$}
\end{overpic}}} $
    \caption{Two tetrahedra glued on top of prism 2.} \label{fig:ontopofprism2}
    \end{subfigure}
    \caption{Replacing four tetrahedra per crossing circle. Here, the boldface numbers index the ideal tetrahedra.}
    \label{fig:triangulationtwoprism}
\end{figure}
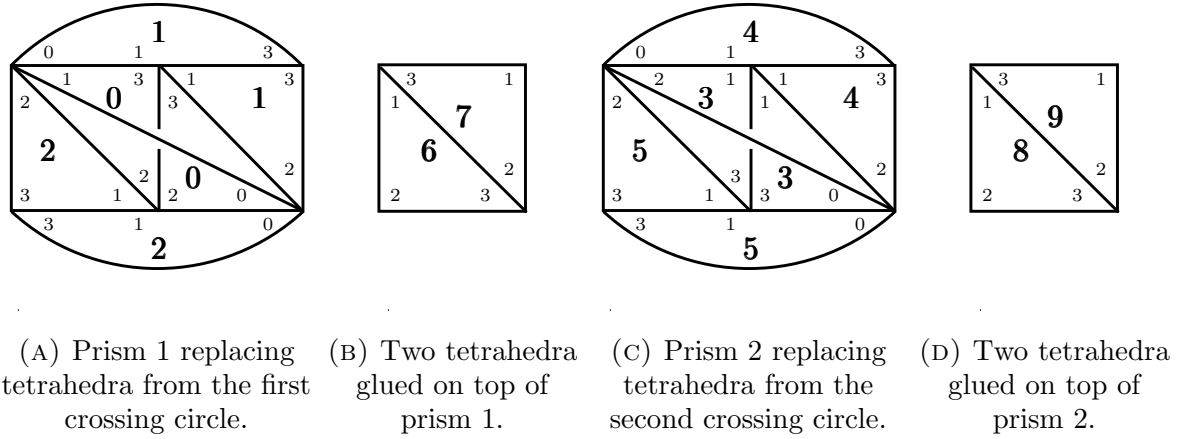

When both crossing circle cusps are replaced by a triangulated prism capped off by two tetrahedra, the cusps are as shown in Figure~\ref{fig:triangulationtwoprism}, and the triangulation is given in Regina notation in Table~\ref{Tab:BorroTrianTwoPrisms}.
The new triangulation of the remaining knot strand cusp neighbourhood is illustrated in Figure~ \ref{fig:cuspntwoprismsml}. Note the two shaded hexagon regions: The edge at the centre of the hexagon runs to a crossing circle cusp.

\begin{table}
    \centering
    \begin{tabular}{c||c|c|c|c}
        $\Delta$  & 012 & 013 & 023 & 123 \\
       \hline \hline 
       $\Delta_0=0$ & 2(021) & 1(201) & 3(130) & 5(213) \\ \hline 
       $\Delta_1=1$ & 0(130) & 2(310) & 7(321) & 4(213) \\ \hline
       $\Delta_2=2$ & 0(021) & 1(310) & 6(312) & 3(321) \\ \hline
       $\Delta_3=3$ & 4(210) & 0(302) & 5(021) & 2(321) \\ \hline  
       $\Delta_4=4$ & 3(210) & 5(310) & 9(321) & 1(213) \\ \hline
       $\Delta_5=5$ & 3(032) & 4(310) & 8(312) & 0(213)  \\ \hline
        6 & 7(012) & 7(032) & 7(031) & 2(230) \\ \hline
        7 & 6(012) & 6(032) & 6(031) & 1(320) \\ \hline
        8 & 9(012) & 9(032) & 9(031) & 5(230) \\ \hline
        9 & 8(012) & 8(032) & 8(031) & 4(320)
    \end{tabular}
 \caption{10-tetrahedron triangulation of the Borromean rings complement. The labels $\Delta_n$ are used in Section~\ref{Sec:GeomTriang}. }
 \label{Tab:BorroTrianTwoPrisms}
\end{table}

\begin{figure}
   \centering
   $\vcenter{\hbox{\begin{overpic}[scale=1.2]{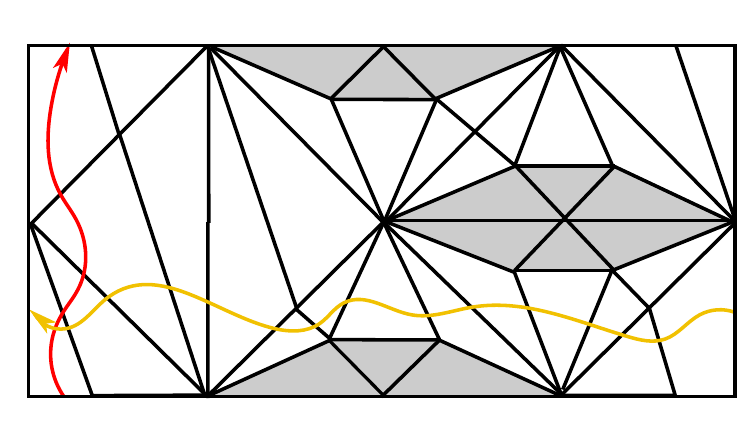}
    \put(14, 123){\contour{black}{\textcolor{white}{${\mathcal{O}_0^1}$}}}
    \put(113, 30){\contour{black}{\textcolor{white}{${\mathcal{O}_0^2}$}}}
    \put(113, 220)
    {\contour{black}{\textcolor{white}{${\mathcal{O}_0^2}$}}}
    \put(62, 173){\contour{black}{\textcolor{white}{$E_0$}}}
     \put(370, 72){\contour{black}{\textcolor{white}{$E_1$}}}
    \put(412, 125){\contour{black}{\textcolor{white}{${\mathcal{O}_0^1}$}}}
    \put(190, 193)
    {\contour{black}{\textcolor{white}{$\mathcal{P}^2$}}}
    \put(242, 193)
    {\contour{black}{\textcolor{white}{${\mathcal{O}_1^2}$}}}
    \put(290, 156)
    {\contour{black}{\textcolor{white}{${\mathcal{O}_1^1}$}}}
    \put(348, 155)
    {\contour{black}{\textcolor{white}{$\mathcal{P}^1$}}}
   \put(20, 40){${\pmb 2}(1)$}
    \put(20.5, 85){$0$}
    \put(55, 44){${\pmb 0}(2)$}
     \put(35, 88){$0$}
      \put(90, 30){$3$}
      \put(25, 200){${\pmb 2}(1)$}
       \put(53, 177){$2$}
        \put(21, 140){$3$}
        
         \put(65, 213){${\pmb 0}(2)$}
         \put(95, 215){$3$}
          \put(70, 189){$1$}
         \put(170, 214){\fontsize{8}{8}${\pmb 6}(2)$} 
    \put(205, 218){\fontsize{8}{8}$0$}
    \put(150, 218){\fontsize{8}{8}$1$}
    \put(188, 203){\fontsize{8}{8}$3$}
   \put(255, 214){\fontsize{8}{8}${\pmb 6}(1)$}
    \put(235, 218){\fontsize{8}{8}$0$}
    \put(290, 218){\fontsize{8}{8}$2$}
    \put(250, 204){\fontsize{8}{8}$3$}
    %
    \put(215, 207){\fontsize{8}{8}${\pmb 7}(3)$}
     \put(219.5, 216){\fontsize{8}{8}$0$}
     \put(234, 199){\fontsize{8}{8}$2$}
     \put(204, 199){\fontsize{8}{8}$1$}
      \put(170, 180){\fontsize{8}{8}${\pmb 2}(3)$}
      \put(148, 203){\fontsize{8}{8}$2$}
    \put(184, 189){\fontsize{8}{8}$0$}
    \put(200, 150){\fontsize{8}{8}$1$}
    \put(215, 170){\fontsize{8}{8}${\pmb 1}(0)$}
     \put(220, 140){\fontsize{8}{8}$1$}
      \put(202, 185){\fontsize{8}{8}$3$}
      \put(240, 185){\fontsize{8}{8}$2$}
      \put(245, 168){\fontsize{8}{8}${\pmb 0}(1)$}
      \put(240, 150){\fontsize{8}{8}$3$}
    \put(252, 181){\fontsize{8}{8}$0$}
    \put(263, 173){\fontsize{8}{8}$2$}
    \put(272, 196){\fontsize{8}{8}${\pmb 2}(2)$}
    \put(272, 182){\fontsize{8}{8}$1$}
    \put(264, 192){\fontsize{8}{8}$0$}
    \put(293, 205){\fontsize{8}{8}$3$}
    %
  \put(170, 30){\fontsize{8}{8}${\pmb 7}(2)$}
    \put(147, 27){\fontsize{8}{8}$1$}
    \put(207, 27){\fontsize{8}{8}$0$}
    \put(187, 47){\fontsize{8}{8}$3$}
 \put(159, 54){\fontsize{8}{8}${\pmb 1}(2)$}  
 \put(147, 43){\fontsize{8}{8}$3$}
 \put(179, 56){\fontsize{8}{8}$0$}
 \put(169, 65){\fontsize{8}{8}$1$}
 %
 \put(255, 30){\fontsize{8}{8}${\pmb 7}(1)$}
  \put(253, 47){\fontsize{8}{8}$3$}
  \put(233, 27){\fontsize{8}{8}$0$}
  \put(295, 27){\fontsize{8}{8}$2$}
  \put(217, 70){\fontsize{8}{8}${\pmb 2}(0)$}
   \put(242, 60){\fontsize{8}{8}$3$}
   \put(199, 60){\fontsize{8}{8}$2$}
   \put(219.5, 105){\fontsize{8}{8}$1$}
   \put(255, 70){\fontsize{8}{8}${\pmb 1}(3)$} 
   \put(236, 100){\fontsize{8}{8}$1$}
    \put(256, 58){\fontsize{8}{8}$0$}
    \put(295, 40){\fontsize{8}{8}$2$}
  \put(183, 78){\fontsize{8}{8}${\pmb 0}(0)$}  
  \put(202, 99){\fontsize{8}{8}$2$}
\put(186, 65){\fontsize{8}{8}$1$}
\put(176, 73){\fontsize{8}{8}$3$}
 \put(217, 42){\fontsize{8}{8}${\pmb 6}(3)$}
 \put(238, 49){\fontsize{8}{8}$2$}
 \put(200, 49){\fontsize{8}{8}$1$}
 \put(220, 30){\fontsize{8}{8}$0$}
 \put(358, 35){${\pmb 0}(3)$}
 \put(368, 55){$0$}
  \put(342, 28){$2$}
 \put(400, 40){${\pmb 1}(1)$}
  \put(413, 99){$3$}
  \put(384, 70){$2$}
   \put(358, 200){${\pmb 0}(3)$}
   \put(340, 215){$2$}
   \put(400, 155){$1$}
   \put(402, 210){${\pmb 1}(1)$}
   \put(415, 165){$0$}
   \put(55, 115){${\pmb 3}(3)$}
    \put(32, 120){$1$}
    \put(62, 157){$2$}
     \put(93, 58){$0$}
      \put(89, 150){${\pmb 5}(1)$}
      \put(110, 64){$0$}
       \put(106, 200){$3$}
        \put(77, 170){$2$}
     \put(128, 115){${\pmb 4}(1)$}
       \put(130, 55){$3$}
        \put(123, 185){$0$}
         \put(155, 75){$2$}
     \put(157, 150){${\pmb 3}(1)$} 
       \put(172, 87){$0$}
      \put(135, 195){$2$}
      \put(205, 123){$3$}
      \put(278, 132){\fontsize{8}{8}${\pmb 8}(1)$}
      \put(310, 129){\fontsize{8}{8}$0$}
       \put(250, 129){\fontsize{8}{8}$2$}
       \put(295, 146){\fontsize{8}{8}$3$}
         \put(360, 132){\fontsize{8}{8}${\pmb 8}(2)$}
        \put(393, 129){\fontsize{8}{8}$1$}
       \put(337, 129){\fontsize{8}{8}$0$}
	\put(355, 147){\fontsize{8}{8}$3$}
       %
     \put(278, 116){\fontsize{8}{8}${\pmb 9}(1)$}
      \put(310, 117){\fontsize{8}{8}$0$}
       \put(250, 117){\fontsize{8}{8}$2$}
	\put(293, 99){\fontsize{8}{8}$3$}
      %
      \put(360, 116){\fontsize{8}{8}${\pmb 9}(2)$}
        \put(393, 117){\fontsize{8}{8}$1$}
       \put(339, 117){\fontsize{8}{8}$0$}
\put(354, 101){\fontsize{8}{8}$3$}
       %
        \put(320, 145){\fontsize{8}{8}${\pmb 9}(3)$}
	  \put(324.5, 132){\fontsize{8}{8}$0$}
  	\put(310, 148){\fontsize{8}{8}$2$}
	  \put(340, 149){\fontsize{8}{8}$1$}
         \put(320, 105){\fontsize{8}{8}${\pmb 8}(3)$}
	  \put(324.5, 115){\fontsize{8}{8}$0$}
  	\put(308, 99){\fontsize{8}{8}$2$}
	  \put(340, 99){\fontsize{8}{8}$1$}
          \put(260, 152){\fontsize{8}{8}${\pmb 5}(2)$}
	       \put(248, 140){\fontsize{8}{8}$3$}
	   \put(283, 156){\fontsize{8}{8}$0$}
	   \put(273, 167){\fontsize{8}{8}$1$}
        \put(277, 80){\fontsize{8}{8}${\pmb 4}(3)$}  
	 \put(308, 45){\fontsize{8}{8}$1$}
	 \put(292, 89){\fontsize{8}{8}$0$}
	\put(240, 109){\fontsize{8}{8}$2$}
       \put(320, 75){\fontsize{8}{8}${\pmb 5}(0)$}
	 \put(303, 88){\fontsize{8}{8}$3$}
	\put(343, 88){\fontsize{8}{8}$2$}
	\put(323, 40){\fontsize{8}{8}$1$}
      \put(346, 65){\fontsize{8}{8}${\pmb 3}(0)$} 
	\put(338, 45){\fontsize{8}{8}$3$}
	\put(353, 85){\fontsize{8}{8}$2$}
	\put(363, 73){\fontsize{8}{8}$1$}
        \put(380, 100){\fontsize{8}{8}${\pmb 4}(2)$}
        \put(400, 108){\fontsize{8}{8}$3$}
        \put(361, 92){\fontsize{8}{8}$0$}
        \put(373, 83){\fontsize{8}{8}$1$}
         \put(289, 182){\fontsize{8}{8}${\pmb 3}(2)$}
         \put(309, 205){\fontsize{8}{8}$1$}
          \put(281, 175){\fontsize{8}{8}$3$}
          \put(293, 166){\fontsize{8}{8}$0$}
\put(320, 170){\fontsize{8}{8}${\pmb 4}(0)$}
 \put(306, 163){\fontsize{8}{8}$2$}
  \put(343, 162){\fontsize{8}{8}$3$}
\put(322, 205){\fontsize{8}{8}$1$}
\put(358, 170){\fontsize{8}{8}${\pmb 5}(3)$}
\put(335, 205){\fontsize{8}{8}$1$}
 \put(362, 159){\fontsize{8}{8}$0$}
 \put(395, 142){\fontsize{8}{8}$2$}
\end{overpic}}} $

    \caption{Cusp neighbourhood with chosen meridian and longitude. The shaded triangles correspond to the four tetrahedra that have three ideal vertices in the knot strand cusp, and their fourth vertex in a crossing circle. Such tetrahedra form two hexagons, one for each crossing circle.}
    \label{fig:cuspntwoprismsml}
\end{figure}

We explain notation. There are two edges of the triangulation that are disjoint from the hexagons, labelled $E_0$ and $E_1$ in Figure~\ref{fig:cuspntwoprismsml}. 
If we remove the two tetrahedra ``caps'' of Figure~\ref{fig:triangulationtwoprism}~(B) and ~(D), we obtain a triangulation with boundary two 1-punctured tori, corresponding to the two crossing circle cusps. This removes the shaded triangles of Figure~\ref{fig:cuspntwoprismsml}.

Each 1-punctured torus has an initial triangulation with edges corresponding to the slopes $\mu_0=1/0$, $\lambda_0=0/1$, and $\delta_0=-1/1$, where $\mu_0$ is a meridian of the crossing circle, equal to the intersection of a neighbourhood of the crossing circle with the projection plane, $\lambda_0$ is a homological longitude, equal to the intersection of the shaded disc in Figure~\ref{fig:idealpolyhedronlabelled}(A), and $\delta_0$ is a diagonal with negative slope. In the first 1-punctured torus, the slope $1/0$ is identified to an edge of the triangulation of the Borromean rings complement that we label $\calO_0^1$. The slope $-1/1$ is identified to an edge labelled $\calO_1^1$, and the slope $0/1$ to an edge labeled $\calP^1$. Here the superscripts indicate the crossing circle cusp. Similarly on the second 1-punctured torus, the slope $1/0$ is identified to the edge $\calO_0^2$, the slope $-1/1$ to edge $\calO_1^2$, and slope $0/1$ to $\calP^2$. 


\section{Dehn filling}\label{Sec:DehnFillingSolidTori}

Consider a complement of a link in $S^3$, $M = S^3 - L$, let $K$ be an unknotted link component of $L$, and let $\mu$ and $\lambda$ be the standard meridian and homological longitude of the tubular neighbourhood of $K$. Then the result of Dehn filling $K$ along any slope of the form $\mu + n \lambda $ for $n \in \mathbb{Z}$ yields a complement of a link in $S^3$. The slope may also be written as $1/n$.

\subsection{Layered solid tori.}
A layered solid torus is a triangulated solid torus whose boundary is a triangulated 1-punctured torus, triangulated by two ideal triangles; see~\cite{JRlayered}. It can be built from the outside in using the Farey graph, which encodes all triangulations of 1-punctured tori; see \cite{GScanonicaltriangulation}. We review the construction briefly.

The \textit{Farey triangulation} is visualised by viewing $\mathbb{H}^2$ in the disc model with points $\QQ\cup \{1/0\}$ labeled on $\bdy_\infty \HH^2$.
We place the antipodal pairs $\{0/1, 1/0\}$ and $\{1/1, -1/1\}$ in $\partial_\infty \mathbb{H}^2$ on a horizontal and vertical line through the centre of the disc, respectively. See Figure~\ref{fig:fareygraphexample} for an illustration. An edge of the Farey triangulation is an ideal geodesic drawn between each pair of points $ a/b, c/d \in \mathbb{Q} \cup \{ 1/0 \} \subset \partial_\infty \mathbb{H}^2$ whenever $|ab-bc| = 1$.

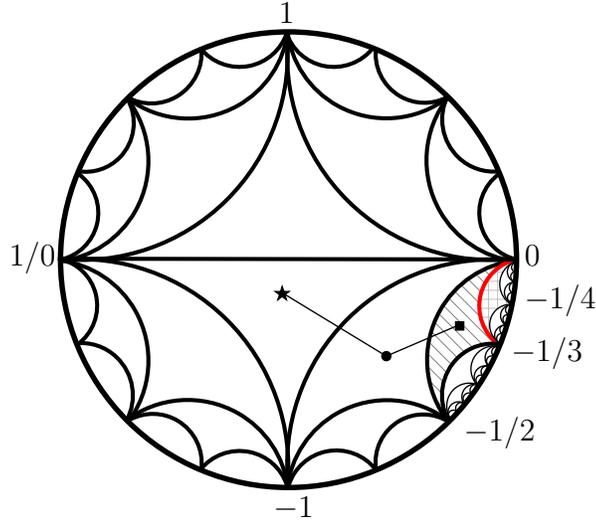
\begin{figure}
\centering
   $\vcenter{\hbox{\begin{overpic}[scale=.6]{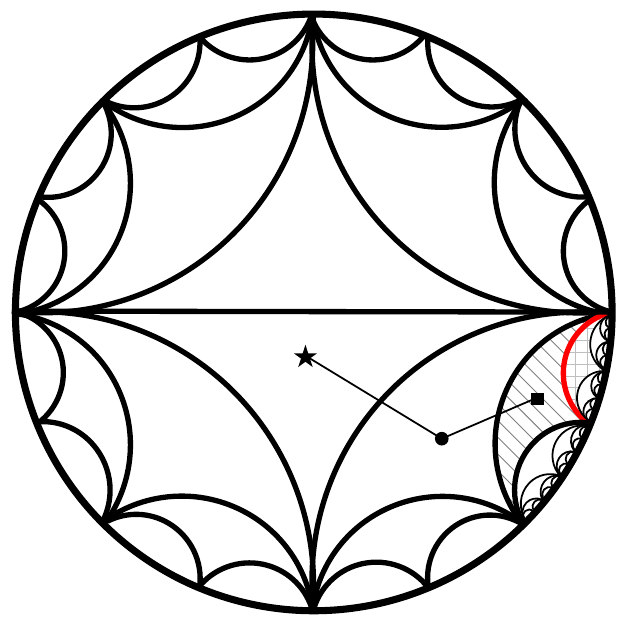}
\put(85,-7){$-1$}
\put(-15,88){$1/0$}
\put(180,88){$0$}
\put(87,180){$1$}
\put(180,72){$-1/4$}
\put(175,52){$-1/3$}
\put(157,22){$-1/2$}
\end{overpic}}} $
    \caption{Example of a walk through the Farey graph to construct a layered solid torus with Dehn filling slope $-1/4$.}
    \label{fig:fareygraphexample}
\end{figure}

On the torus, points of $\QQ\cup\{ 1/0\}$ represent slopes. For a 1-punctured torus, we require all slopes to meet the puncture point. Then triangles of the Farey triangulation correspond to triangulations of the 1-punctured torus, since two slopes with $|ad-bc|=1$ have geometric intersection number equal to one. A step in the Farey graph from one triangle $T_i$ to an adjacent triangle $T_{i+1}$ corresponds to changing the 1-punctured torus triangulation: Remove the slope represented by the vertex of $T_i$ that is not incident with the shared edge between $T_i$ and $T_{i+1}$, and replace it by the slope represented by the new vertex of $T_{i+1}$. The result is a diagonal exchange.

In three dimensions, switching slopes can be obtained by layering on a tetrahedron: a step in the Farey graph corresponds to placing a tetrahedron onto the triangulated torus such that the bottom two faces are covered by the tetrahedron, and the top two faces have swapped the edge.

When we Dehn fill, we start with a 3-dimensional triangulation with boundary a 1-punctured torus, and some initial 2-dimensional triangulation on this boundary. 
The slope of the Dehn filling $s$ will be some distance from this initial triangulation in the Farey triangulation. Starting at the centre of the initial triangulation in $\HH^2$, take a geodesic path through the Farey triangulation to the nearest triangle with one of the slopes the desired slope to be filled. Each time this geodesic passes out of one Farey triangle into another, the triangulation changes, and this is realised by layering on a tetrahedron for all but the very last triangle. Before stepping into the final triangle of the Farey graph, the desired meridian forms the diagonal that would be exchanged if we added one more tetrahedron. Instead of adding this tetrahedron, fold to identify the two triangles. This makes the desired meridian homotopically trivial; see, for example, \cite[Section~3]{HPhighlytwisted}.

\subsection{Construction of (conjecturally) minimal triangulation}\label{Sec:mintriangulation}
We will now construct explicitly a triangulation of $K(p,q)$, for $p,q\geq 2$, denoted by $\calT_{K(p,q)}$ and prove that it is geometric. 

Starting with the 10-tetrahedra triangulation of the Borromean rings complement of Section~\ref{Sec:10TetB}, fix one of the crossing circle cusps and remove the two tetrahedra meeting that cusp, leaving a triangulation with a boundary component consisting of a 1-punctured torus.

By construction, the slopes on the boundary are the standard meridian $1/0$, the homological longitude $0/1$, and the slope $-1/1$. Thus to attach a layered solid torus to the first crossing circle cusp, along slope $-1/p$, we follow the procedure:

\begin{enumerate}
\item If $p>2$, start in the triangle in the Farey triangulation with vertices $1/0$, $-1/1$, $0/1$ and step into the Farey triangle with vertices $-1/1$, $-1/2$, $0/1$. Layer a tetrahedron $\Delta_0^1$ onto the 1-punctured torus boundary corresponding to the first crossing circle. The edges of $\Delta_0^1$ are glued to the edge $\calO_0^1$ corresponding to slope $1/0$, the edge $\calO_1^1$ corresponding to slope $-1/1$, and $\calP^1$ corresponding to $0/1$. The edge $\calO_0^1$ is covered, and a new edge $\calO_2^1$ is produced. 

\item Proceed through the Farey triangulation through triangles with vertices $-1/(j-1)$, $-1/j$, $0/1$.  At each step add a tetrahedron $\Delta_{j-2}^1$ with edges corresponding to the previous and current Farey triangle. The edge $\calO_{j-2}^1$ will be covered, the edge $\calO_j^1$ produced, and the other edges of the tetrahedra will be glued to $\calO_{j-1}^1$ and $\calP^1$. 
 
\item Repeat until reaching the Farey triangle with slopes $-1/(p-2)$, $-1/(p-1)$, $0/1$. The final tetrahedron is $\Delta_{p-3}^1$ (since the edge $\calO_{p-3}^1$ is covered when it is attached). 

\item Fold along the edge $\calO_{p-2}^1$ of the boundary triangle associated to the slope $-1/(p-2)$. As a result, the diagonal slope $-1/p$ will become homotopically trivial, yielding the appropriate Dehn filling. Note that because of the folding, the edge $\calO_{p-1}^1$ is identified with $\calP^1$. 
\end{enumerate}

In the cusp triangulation shown in Figure~\ref{fig:cuspntwoprismsml}, which corresponds to the knot strand cusp that is not filled, 
this will remove one of the shaded hexagons and replace it by a hexagon corresponding to the ideal point of the newly attached layered solid torus. Figure~\ref{fig:MinConstructionHexagon}~(A) shows how the hexagon changes when layering in the first two tetrahedra. Figure~\ref{fig:MinConstructionHexagon}~(B) shows the effect of folding after layering in those two tetrahedra; this corresponds to Dehn filling slope $-1/4$. 

\begin{figure}
  \centering
    \begin{subfigure}{.47\textwidth}
    \centering
 \resizebox{\columnwidth}{!}{$\vcenter{\hbox{\begin{overpic}[scale=1.2]{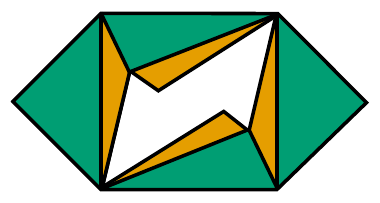}
     \put(25,55){$\Delta_{0}^1$}
    \put(80,95){$\Delta_0^1$}
    \put(170,55){$\Delta_0^1$}
    \put(115,15){$\Delta_0^1$}
 \put(0, 56){\contour{white}{${\mathcal{O}_0^1}$}}
  \put(50, 107){\contour{white}{${\mathcal{O}_1^1}$}}
 \put(69, 74){\contour{white}{${\mathcal{O}_2^1}$}}
 \put(154, 107){\contour{white}{${\mathcal{P}_{1}}$}}
  \put(85, 61){\contour{white}{${\mathcal{O}_3^1}$}}
\end{overpic}}}$}
    \caption{layering.} \label{fig:layering}
    \end{subfigure}  \quad
  \begin{subfigure}{.47\textwidth}
    \centering
 \resizebox{\columnwidth}{!}{$\vcenter{\hbox{\begin{overpic}[scale=1.2]{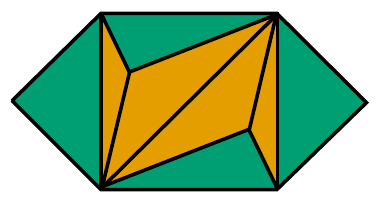}
     \put(25,55){$\Delta_0^1$}
    \put(80,95){$\Delta_0^1$}
    \put(170,55){$\Delta_0^1$}
    \put(115,15){$\Delta_0^1$}
\end{overpic}}}$}
    \caption{Folding.} \label{fig:folding}
    \end{subfigure}
  
    \caption{Layering in the 1-punctured torus.}
    \label{fig:MinConstructionHexagon}
\end{figure}

Repeat for the second crossing circle. The process attaches tetrahedra $\Delta_0^2, \dots, \Delta_{q-2}^2$, and introduces edges $\calO_2^2, \dots, \calO_{q-2}^2$. The result is the triangulation of the Dehn filling $\calT_{K(p,q)}$.

Figure~\ref{fig:exampleK33} shows an example for $p=3$ and $q=3$. We layer one tetrahedron per crossing circle, then fold. The resulting triangulation is shown in the figure.
\begin{figure}[ht]
\centering
$\vcenter{\hbox{\begin{overpic}{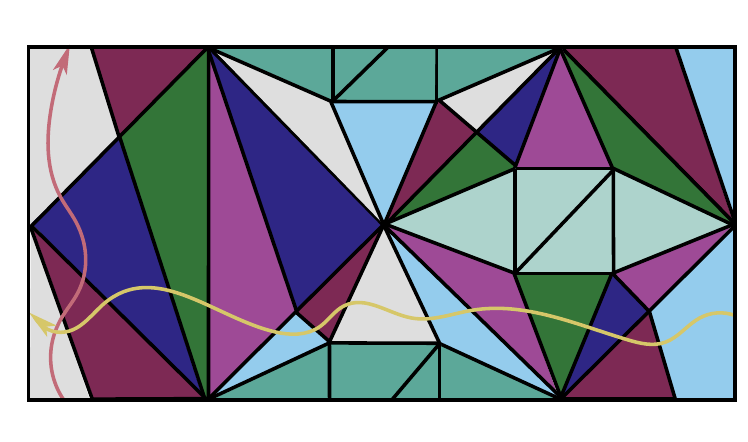}
\end{overpic}}} $
 \caption{Triangulation of $S^3 - K_{3,3}$ with meridian and longitude.}
    \label{fig:exampleK33}
\end{figure}

\subsection{Construction of canonical triangulation}\label{Sec:CanonicalConstruction}

We now construct a different triangulation of the double twist knots. This construction yields the canonical triangulation, also known as the Sakuma--Weeks triangulation~\cite{SWcanonicaltriangulations}. The fact that the Sakuma--Weeks triangulation is canonical follows from work of Gu{\'e}ritaud and Futer~\cite[Appendix~A]{GueritaudFuter:Canonical}; see also \cite{ASWY}. Our construction is described differently, via Dehn filling the Borromean rings, triangulated as in Section~\ref{Sec:8TetB}. Ham has recently shown that all 3-manifolds obtained by more general Dehn filling this triangulation of the Borromean rings complement are canonical~\cite{ham2023canonical}; ours is an instance of her construction. Since the canonical triangulation is unique, these triangulations must agree. 

Start with the 8-tetrahedra triangulation of the Borromean rings complement of Section~\ref{Sec:8TetB}, and remove four tetrahedra corresponding to one of the crossing circle cusps. Recall this gives a manifold with boundary a 2-punctured torus.

Consider the slope $-1/p = \mu - p \lambda$ on the twice punctured torus where $\mu$ and $\lambda$ are the meridian and homological longitude. Consider the lift of the twice punctured torus to the infinite cover $\mathbb{R}^2$, where the punctures are located at the integral points $\mathbb{Z}^2 \subset \mathbb{R}^2$. In this cover, $\mu$ lifts to the line segment from $(0, 0)$ to $(0, 1)$, and $\lambda$ lifts to the line segment from $(0, 0)$ to $(2, 0)$.  Then the slope $ \mu - p \lambda$ lifts to an arc with end points at $(0, 0)$ and $(-2p, 1)$. Furthermore, the lift of the slope $-1/p$ will only meet the integral points $\mathbb{Z}^2$ at its endpoints.

The solid torus that we glue in this setting is the double cover of a layered solid torus:

\begin{lemma}[LST double cover]\label{lemma:HP}
Let $m = -1/p = -\mu+p \lambda$ be a slope on a torus with generators $\mu, \lambda$ and $p \notin \{0, \pm 1\}$. Construct a layered solid torus $X$ by beginning in the Farey triangle with vertices $(1/0, 0/1, -1/1)$ and step to the triangle with slope $-1/2p$.
Then, the double cover of $X$, denoted $Y$, satisfies the following properties.
\begin{enumerate}
\item The boundary of $Y$ is a 2-punctured torus that is triangulated by four ideal triangles. The basis slope $\lambda$ lifts to run from $(0,0)$ to $(2, 0)$ in $\mathbb{R}^2$, and projects to run twice around the slope $0/1$ in $\partial X$. The slope $\mu$ lifts to run from $(0,0)$ to $(0,1)$ in $\mathbb{R}^2$. Since $m$ is negative, the diagonals of the triangulation of $\partial Y$ have negative slope.
\item The meridian of $Y$ is the slope $m= - \mu +p \lambda$. 
\end{enumerate} 
\end{lemma}

\begin{proof}
This is \cite[Lemma~7.1]{HPhighlytwisted}, with $\ell$ replaced by $1$. 
\end{proof}

Consider the knot strand cusp neighbourhood of the Borromean rings illustrated in Figure~\ref{fig:triangulation}~(B). By removing the tetrahedra formed from the first crossing circle, we obtain a manifold $M$ with a twice-punctured torus as its boundary. By Lemma~\ref{lemma:HP} we construct the double cover $Y_1$ of a layered solid torus $X_1$. To obtain meridian $-1/p$, construct $X_1$ by starting in the Farey triangle with slopes $(1/0, -1/1, 0/1)$ and ending in the Farey triangle $(-1/(2p-2), -1/(2p-1), 0/1)$, and then folding. Take $Y_1$ to be the double cover. Repeat the process for the second crossing circle and slope $-1/q$ to obtain the double cover $Y_2$ of a solid torus $X_2$, and glue $Y_1$ and $Y_2$ to form a triangulation of the complement of the knot $K(p,q)$.

Two edges on the boundary of the solid torus $Y_1$ cover the slope $1/0$ in the layered solid torus $X_1$ of Lemma~\ref{lemma:HP}. These both glue to the same edge $\calP^2$ in the knot complement. The two slopes of $Y_2$ projecting to $0/1$ in $X_2$ also glue to $\calP^2$. Slopes on $Y_2$ projecting to $1/0$ in $X_2$ and slopes on $Y_1$ projecting to $0/1$ on $X_1$ all glue to edge $\calP^1$ in the knot complement. Finally, the diagonal slopes on $Y_1$ glue to an edge $\mathcal{O}_{1,1}$, and diagonal slopes on $Y_2$ glue to an edge $\mathcal{O}_{1,2}$. Within the double cover $Y_i$ of the solid torus $X_i$, $i=1,2$, there will be two edges corresponding to slope $-1/j$ for $j=1, \dots, p-2$ or $q-2$. In the knot complement, these glue to distinct edges if $j\geq 2$. Denote them by $\calO_{j,1}^i$ and $\calO_{j,2}^i$. Thus superscript indicates cusp, and subscript indicates slope and edge of the double cover. 

Because this is a double cover, each step in the Farey triangulation produces two tetrahedra. Moreover, for slope $-1/p$, we must walk to $-1/(2p)$ in the Farey triangulation. Thus each step that produced only one tetrahedron in the previous construction produces four tetrahedra in this one. 

We illustrate a few steps. The first layering of tetrahedra is shown in Figure~\ref{fig:CanonConstructionFirstLayer}, where $(\pmb i)_k^j$ denotes the tetrahedron glued onto the edge $\mathcal{O}_{i,k}^j$ in the $j^{th}$ cusp, $j=\{1,2\}$. Tetrahedron $(\pmb 0)^j_1$ covers $\mathcal{O}_{0,1}^j$ and tetrahedron $(\pmb 0)_2^j$ covers $\mathcal{O}_{0,2}^j$. Both tetrahedra meet at edges $\mathcal{O}_{1,1}$ and $\mathcal{O}_{1,2}$ and produce new edges: $\mathcal{O}_{2,1}^j$ and $\mathcal{O}_{2,2}^j$, respectively. Edges of each tetrahedron are glued to the edge labelled $\calP_j$. 

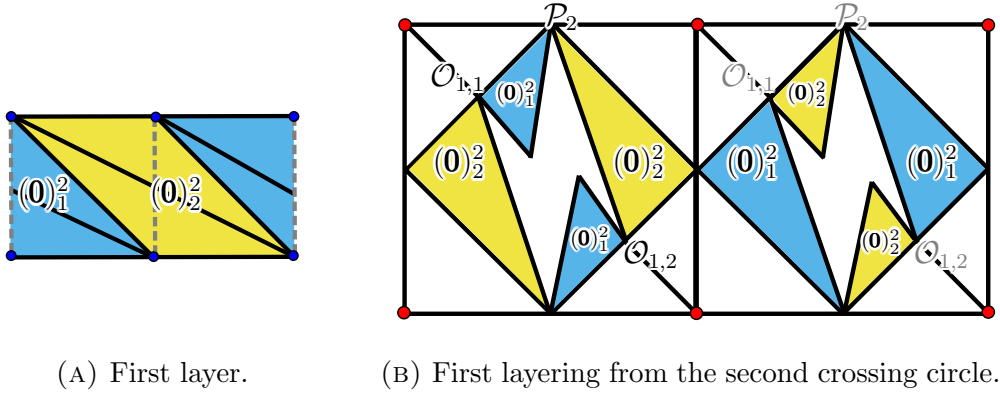
\begin{figure}[ht]
  \centering
    \begin{subfigure}{.35\textwidth}
    \centering
$\vcenter{\hbox{\begin{overpic}{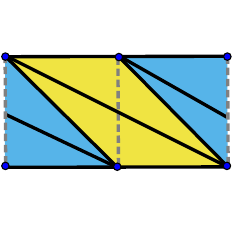}
          \put(5, 53){\contour{white}{$(\pmb 0)_1^2$}}
          \put(55, 53){\contour{white}{$(\pmb 0)_2^2$}}
\end{overpic}}} $
    \caption{First layer.} 
    \end{subfigure}
  \begin{subfigure}{.5\textwidth}
    \centering
$\vcenter{\hbox{\begin{overpic}{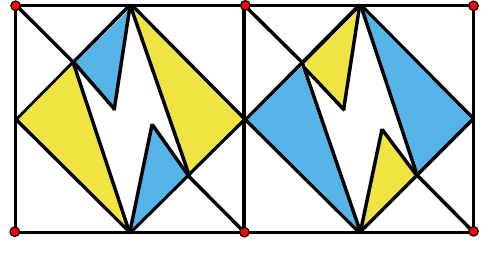}
      \put(17, 97){\contour{white}{$\mathcal{O}_{1,1}$}}
        \put(90, 30){\contour{white}{$\mathcal{O}_{1,2}$}}
        \put(200, 30){\contour{white}{\textcolor{gray}{$\mathcal{O}_{1,2}$}}}
       \put(127, 97){\contour{white}{\textcolor{gray}{$\mathcal{O}_{1,1}$}}}
       \put(18, 65){\contour{white}{$(\pmb 0)_2^2$}}
       \put(86, 65){\contour{white}{$(\pmb 0)_2^2$}}
       \put(129, 65){\contour{white}{$(\pmb 0)_1^2$}}
       \put(197, 65){\contour{white}{$(\pmb 0)_1^2$}}
        \put(152, 92){\fontsize{8}{8}\contour{white}{$(\pmb 0)_2^2$}}
       \put(180, 36){\fontsize{8}{8}\contour{white}{$(\pmb 0)_2^2$}}
        \put(42, 92){\fontsize{8}{8}\contour{white}{$(\pmb 0)_1^2$}}
       \put(70, 37){\fontsize{8}{8}\contour{white}{$(\pmb 0)_1^2$}}
\put(60, 120){\contour{white}{$\mathcal{P}_{2}$}}
\put(170, 120){\contour{white}{\textcolor{gray}{$\mathcal{P}_{2}$}}}
\end{overpic}}} $
    \caption{First layering from the second crossing circle.} 
    \end{subfigure}
    \caption{Layering in the 2-punctured torus.}
    \label{fig:CanonConstructionFirstLayer}
\end{figure}

The second layering is illustrated in Figure~\ref{fig:CanonConstructionSecondLayer}. For $j=\{1,2\}$, tetrahedron $(\pmb 1)^j_1$ is layered over $\mathcal{O}_{1,1}^j$ and tetrahedron $(\pmb 1)_2^j$ is layered over $\mathcal{O}_{1,2}^j$. Both tetrahedra meet at edges $\mathcal{O}_{2,1}^j$ and $\mathcal{O}_{2,2}^j$ and produce new edges: $\mathcal{O}_{3,1}^j$ and $\mathcal{O}_{3,2}^j$, respectively. Edges of the tetrahedra are glued to the edge labelled $\mathcal{P}_j$.

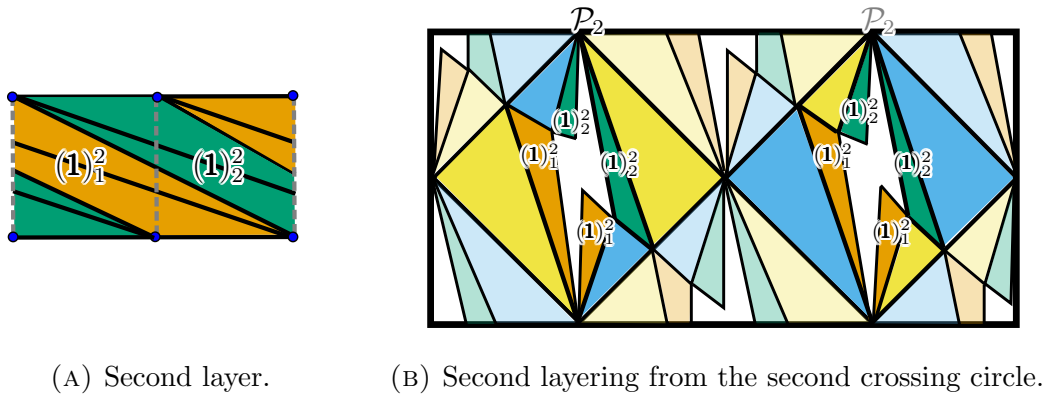
\begin{figure}[ht]
   \centering
    \begin{subfigure}{.33\textwidth}
    \centering
$\vcenter{\hbox{\begin{overpic}{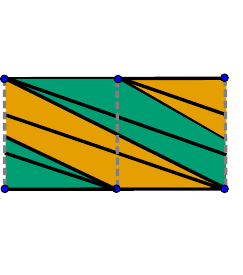}
      \put(18, 65){\contour{white}{$(\pmb 1)_1^2$}}
       \put(70, 65){\contour{white}{$(\pmb 1)_2^2$}}
\end{overpic}}} $ 
    \caption{Second layer.} 
    \end{subfigure}
  \begin{subfigure}{.55\textwidth}
    \centering
$\vcenter{\hbox{\begin{overpic}{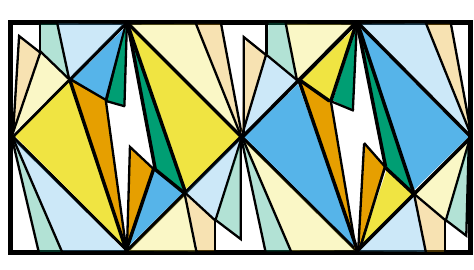}    
      \put(39, 70){\contour{white}{\fontsize{8}{8}$(\pmb 1)_1^2$}}
        \put(60, 41){\contour{white}{\fontsize{8}{8}$(\pmb 1)_1^2$}}
  \put(150, 70){\contour{white}{\fontsize{8}{8}$(\pmb 1)_1^2$}}
        \put(172, 41){\contour{white}{\fontsize{8}{8}$(\pmb 1)_1^2$}}
\put(58, 120){\contour{white}{$\mathcal{P}_{2}$}}
\put(168, 120){\contour{white}{\textcolor{gray}{$\mathcal{P}_{2}$}}}
  
 \put(51, 83){\contour{white}{\fontsize{8}{8}$(\pmb 1)_2^2$}}
       \put(160, 87){\contour{white}{\fontsize{8}{8}$(\pmb 1)_2^2$}}
 \put(180, 67){\contour{white}{\fontsize{8}{8}$(\pmb 1)_2^2$}}
 \put(69, 67){\contour{white}{\fontsize{8}{8}$(\pmb 1)_2^2$}}
\end{overpic}}} $
    \caption{Second layering from the second crossing circle.}
    \end{subfigure}

    \caption{Layering in the twice-punctured torus.}
    \label{fig:CanonConstructionSecondLayer}
\end{figure}

The layering pattern continues until tetrahedra $(\pmb{2p-3})_1^j$ and $(\pmb{2p-3})_2^j$ are layered on, each with an edge gluing to $\calP_j$. 

Once tetrahedra $(\pmb{2p-3})_1^j$ and $(\pmb{2p-3})_2^j$ are layered on, fold the triangular faces of the boundary torus along two of its edges to make the curve of slope $-1/p$ homotopically trivial. For cusp $j \in \{ 1,2\}$, the triangular faces of the boundary torus are faces of tetrahedra $(\pmb{2p-3})_1^j$ and $(\pmb{2p-3})_2^j$ with edges $\{ \mathcal{O}_{2p-2,i}^j, \mathcal{O}_{2p-1,i}^j, \mathcal{P}_j \}$ for $i=1$ or $2$, respectively. We fold across the edges $ \mathcal{O}_{2p-2,1}^j$ and $\mathcal{O}_{2p-2,2}^j$. As a consequence, edges $\mathcal{O}_{2p-1,1}^j$ and $\mathcal{O}_{2p-1,2}^j$ are identified with edge $\mathcal{P}_j$. An illustration for $p=2$ is given in Figure~\ref{fig:CanonFoldingExample}.

\begin{figure}[ht]
\centering
$\vcenter{\hbox{\begin{overpic}{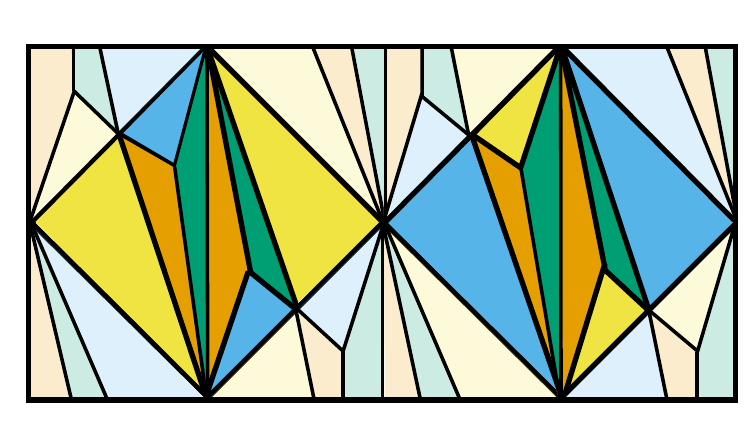}
\end{overpic}}} $
 \caption{Example of folding for $p=2$.}
    \label{fig:CanonFoldingExample}
\end{figure}

\section{Geometric triangulations}\label{Sec:GeomTriang}

Now return to the triangulation $\calT_{K(p,q)}$ of Section~\ref{Sec:mintriangulation}. In this section, we prove that $\calT_{K(p,q)}$ is geometric, meaning the tetrahedra in the construction inherit positively oriented convex hyperbolic structures from the complete hyperbolic structure on the knot complement. The main result is the following.

\begin{theorem}\label{Thm:geometric}
For every $p,q\geq 2$, the ideal triangulation $\calT_{K(p,q)}$ of the complement of $K(p,q)$ constructed in Section~\ref{Sec:mintriangulation} is geometric.
\end{theorem}

We prove this theorem using angle structures. For more background, see Futer and Gu\'eritaud~\cite{FuterGueritaud:Angled}. 

We will require a labelled triangulation, in the following sense.

\begin{definition}\label{Def:LabelledTriang}
An \emph{oriented labelling} of an ideal tetrahedron is a labelling of its four ideal vertices with the numbers $0,1,2,3$,
up to oriented homeomorphism preserving edges. The opposite pairs of edges $01$ and $23$ are called $a$-edges, $02$ and $13$ $b$-edges, and $03$ and $12$ $c$-edges. The orientation is such that, as viewed from outside of the tetrahedron, the three incident edges around a vertex are an $a$-, $b$-, and $c$-edge and they are ordered anticlockwise. 

A \emph{labelled triangulation} of $M$ is an oriented ideal triangulation of $M$, where
\begin{enumerate}
\item $M$ consists of ideal tetrahedra with a fixed chosen ordering $\Delta_1, \dots, \Delta_{\frak{n}}$,
\item the edges are labelled with a fixed chosen ordering $E_1, \dots, E_{\frak{n}}$, and 
\item each tetrahedron is given an oriented labelling. 
\end{enumerate}
\end{definition}

Now consider the space of angle structures $\scrA(\calT_{K(p,q)})$ of the ideal triangulation $\calT_{K(p,q)}$. 
For $p,q>2$, a point $\alpha\in \scrA(\calT_{K(p,q)})$ has the form
\begin{align*}
\alpha = &(a_0, b_0, c_0, a_1, b_1, c_1, \cdots, a_5, b_5, c_5, \\
 & a_{0, 1}, b_{0, 1}, c_{0, 1}, \cdots, a_{p-3, 1},b_{p-3, 1},  c_{p-3, 1}, \\
 & a_{0, 2}, b_{0, 2}, c_{0, 2}, \cdots, a_{q-3, 2},b_{q-3, 2},  c_{q-3, 2})
\end{align*}
where $a_*, b_*, c_*$ lie in $(0,\pi)$ and sum to $\pi$, and correspond to angles on edges of the same tetrahedron, labelled to agree with Definition~\ref{Def:LabelledTriang}. The subscripts are as follows. Six of the tetrahedra agree with tetrahedra $0$ through $5$ of Table~\ref{Tab:BorroTrianTwoPrisms}, for the Borromean rings complement. We label the angles corresponding to the $j$-th such tetrahedra by $a_j$, $b_j$, $c_j$. The tetrahedra $\Delta_j^i$, $i=1, 2$, of the construction of Section~\ref{Sec:mintriangulation} have corresponding edges $a_{j,i}$, $b_{j,i}$, and $c_{j,i}$.

Any angle structure satisfies the edge gluing equations, given below, where the equations are labelled by the corresponding edge.
The edge equations for edges outside the layered solid torus may be read off of Figure~\ref{fig:cuspntwoprismsml}, and those within the layered solid torus may be read off of Figure~\ref{fig:MinConstructionHexagon} or its generalisation for higher numbers of tetrahedra. 
\begin{align}
\calO_0^1 \: & :\:  2b_0 + a_1 +b_1 +a_2+b_2+b_3+a_4+a_5+ a_{0, 1}=2 \pi, \label{eqn:w1} \\
\calO_1^1 \: & :\:  b_3 + b_4 + b_5+2c_{0, 1}+a_{1,1} =2 \pi, \label{eqn:w2}\\
\calP^1 \: & :\:  c_4 +c_5+ 2b_{0, 1} + \cdots + 2b_{p-3, 1}+ a_{p-3, 1} =2 \pi, \label{eqn:w5}\\
\calO_k^1 \: & :\:   a_{k-2,1} + 2 c_{k-1, 1}+ a_{k, 1} =2 \pi \quad \text{ for } k=2, \dots, p-3, \label{eqn:w9}\\
\calO_{p-2}^1 \: & :\:  a_{p-4, 1}+2c_{p-3, 1} =2 \pi,\label{eqn:w11}\\
E_0 \: & :\:  c_0 + c_2 + a_3 + c_5 =2 \pi, \label{eqn:w7}\\
E_1 \: & :\:  c_0 + c_1 + a_3 + c_4 =2 \pi, \label{eqn:w8}\\
\calO_0^2 \: & :\:  a_0 +a_1 + a_2+ 2c_3 +a_4+ b_4 +a_5 + b_5+a_{0,2} =2 \pi, \label{eqn:w3}
\end{align}
\begin{align}
\calO_1^2 \: & :\:  a_0 + b_1 + b_2+2c_{0, 2}+a_{1, 2} =2 \pi, \label{eqn:w4}\\
\calP^2 \: & :\:  c_1+c_2+ 2b_{0, 2} + \cdots + 2b_{q-3, 2}+ a_{q-3, 2} =2 \pi, \label{eqn:w6}\\
\calO_j^2 \: & :\:  a_{j-2,2} + 2 c_{j-1, 2}+ a_{j, 2} = 2\pi \quad \text{ for } j=2, \dots, q-3, \label{eqn:w10}\\
\calO_{q-2}^2 \: & :\:   a_{q-4, 2}+2c_{q-3, 2} =2 \pi. \label{eqn:w12}
\end{align}

When $p=2$ and $q>2$, and $(a_0, b_0, c_0, \cdots, a_5, b_5, c_5$, $a_{0, 2}, b_{0, 2}, c_{0, 2}, \cdots, a_{q-3, 2},b_{q-3, 2},  c_{q-3, 2}) \in \scrA(\calT_{K(2,q)})$ satisfy Equations \eqref{eqn:w7} through \eqref{eqn:w12} and the following two equations:
\begin{align}
\calO_0^1 \: & :\:   2b_0 + a_1 +b_1 +a_2+b_2+b_3+a_4+a_5 =2 \pi, \label{eqn:2q1} \\
\calO_1^1 \: & :\:  b_3 + b_4 + b_5+c_4 +c_5 = 2 \pi.\label{eqn:2q2} 
\end{align}

For $p>2$ and $q=2$, a point $(a_0, b_0, c_0, \cdots, a_5, b_5, c_5$, $a_{0, 1}, b_{0, 1}, c_{0, 1}, \cdots, a_{p-3, 1},b_{p-3, 1},  c_{p-3, 1}) \in \scrA(\calT_{K(p,q)})$ satisfies Equations \eqref{eqn:w1} through \eqref{eqn:w8}, and the following two equations:
\begin{align}
\calO_0^2 \: & :\:    a_0 +a_1 + a_2+ 2c_3 +a_4+ b_4 +a_5 + b_5 =2 \pi, \label{eqn:p21}\\
\calO_1^2 \: & :\:  a_0 + b_1 + b_2 +c_1+c_2= 2 \pi.\label{eqn:p22}
\end{align}

For $p=2$ and $q=2$, $(a_0, b_0, c_0, \cdots, a_5, b_5, c_5) \in \scrA(\calT_{K(2,2)})$ satisfies Equations \eqref{eqn:w7}, \eqref{eqn:w8}, \eqref{eqn:2q1}, \eqref{eqn:2q2}, \eqref{eqn:p21}, and \eqref{eqn:p22}.

\begin{lemma}\label{Lem:nonempty}
For $p,q \geq 2$, the set $\scrA(T_{K(p,q)})$ is non-empty.
\end{lemma}

\begin{proof}
For small $\epsilon>0$ and $\delta>0$,  $1<k\leq p-3$, and $1<j\leq q-3$, define $\alpha \in \scrA(\calT_{K(p,q)})$ by:
\[
\begin{pmatrix}
a_{k, 1} \\
b_{k, 1} \\
c_{k, 1} 
\end{pmatrix} = \begin{pmatrix}
 \epsilon(p-k-2)^2 \\
\epsilon \\
\pi - \epsilon ( (p-k-2)^2+1)
\end{pmatrix}, \quad \begin{pmatrix}
a_{j, 2} \\
b_{j, 2} \\
c_{j, 2} 
\end{pmatrix} = \begin{pmatrix}
 \delta(q-j-2)^2 \\
\delta \\
\pi - \delta ( (q-j-2)^2+1)
\end{pmatrix},
\]

\[
\begin{pmatrix}
a_{0, 1} \\
b_{0, 1} \\
c_{0, 1} 
\end{pmatrix} =  \begin{pmatrix}
 \epsilon (p-2)^2 \\
\pi/2 - \epsilon(p-2)^2/2\\
\pi/2  - \epsilon (p-2)^2/2
\end{pmatrix}, \quad  \begin{pmatrix}
a_{0, 2} \\
b_{0, 2} \\
c_{0, 2} 
\end{pmatrix} =  \begin{pmatrix}
 \delta (q-2)^2 \\
\pi/2 -  \delta (q-2)^2/2 \\
\pi/2 -  \delta  (q-2)^2/2
\end{pmatrix},
\]

\[
\begin{pmatrix}
a_{0} \\
b_{0}  \\
c_{0}
\end{pmatrix}
=\begin{pmatrix}
\frac{\pi}{3}  \\
\frac{\pi}{6} -\epsilon(2p-5)+\delta(q^2-6q+9)/2\\
\frac{\pi}{2}+\epsilon(2p-5)-\delta(q^2-6q+9)/2
\end{pmatrix},
\] 

\[
\begin{pmatrix}
a_{1} \\
b_{1} \\
c_{1} 
\end{pmatrix} =  \begin{pmatrix}
 \frac{\pi}{6}  -\delta(q^2-2q-1)/2\\
\frac{\pi}{3} +\delta(2q-5)\\
\frac{\pi}{2} +\delta(q^2-6q+9)/2
\end{pmatrix}, \qquad \begin{pmatrix}
a_{2}\\
b_{2} \\
c_{2}
\end{pmatrix} =  \begin{pmatrix}
 \frac{\pi}{6}  -\delta(q^2-6q+9)/2\\
\frac{\pi}{3} \\
\frac{\pi}{2}  +\delta(q^2-6q+9)/2
\end{pmatrix},
\]

\[
\begin{pmatrix}
a_{3} \\
b_{3} \\
c_{3} 
\end{pmatrix} = \begin{pmatrix}
\frac{\pi}{2}   -\epsilon(p^2-2p-1)/2\\
\frac{\pi}{3} +\epsilon(2p-5)\\
\frac{\pi}{6} +\epsilon(p^2-6p+9)/2
\end{pmatrix}, \quad \text{and} \quad \begin{pmatrix}
a_{4} \\
b_{4} \\
c_{4} 
\end{pmatrix}=\begin{pmatrix}
a_{5} \\
b_{5} \\
c_{5} 
\end{pmatrix} =  \begin{pmatrix}
 \frac{\pi}{6}-\epsilon(p^2-6p+9)/2\\
\frac{\pi}{3} \\
\frac{\pi}{2}+\epsilon(p^2-6p+9)/2
\end{pmatrix}.\]

The following statements are verified by direct substitution.
\begin{enumerate}
\item For $p, q>2$, small $\epsilon$, and small $\delta$, $\alpha$ is a shape structure that satifies all gluing equations \eqref{eqn:w1} through \eqref{eqn:w12}, and thus lies in $\scrA(\calT_{K(p,q)})$. 

\item For $p=2$ and $q>2$, set $\epsilon=0$. Then for small $\delta$, $\alpha$ satisfies all required tetrahedron and edge gluing equations and thus lies in $\scrA(\calT_{K(2,q)})$.
\item Similarly for $q=2$ and $p>2$, set $\delta=0$. Then for small $\epsilon$, $\alpha$ satisfies tetrahedron and edge gluing equations. 
\item For $p=q=2$, set $\epsilon=\delta=0$. Then one checks directly that $\alpha$ lies in $\scrA(\calT_{K(2,2)})$. \qedhere
\end{enumerate}
\end{proof}

A tetrahedron $\Delta$ of a triangulation $\calT$ endowed with an extended shape structure $\alpha \in \overline{\scrA(\calT)}$ is \emph{flat} for $\alpha$ if at least one of the three angles of $\calT$ is zero, and \emph{taut} for $\alpha$ if two angles are zero and the third is $\pi$. In both cases, $\Delta$ has volume equal to zero.

\begin{lemma}\label{lemma:flatimpliestaut}
Suppose for $p,q\geq 2$, the volume functional on $\overline{\scrA(\calT_{K(p,q)})}$ is maximal at $\alpha \in \overline{\scrA(\calT_{K(p,q)})} \setminus \scrA(\calT_{K(p,q)})$. If an angle of $\alpha$ equals $0$, then the other two angles associated to the same tetradedron are $0$ and $\pi$. 
\end{lemma}

\begin{proof}
This follows from \cite[Proposition 7.1]{GueritaudFuter:Canonical}; see also~\cite[Proposition 9.19]{PurcellHyperbolicKnotTheory}.
\end{proof}

\begin{lemma}\label{LSTflat}
Suppose the volume functional on the space of angle structures on a layered solid torus takes its maximum on the boundary. Then this structure consists of only taut tetrahedra.
\end{lemma}

\begin{proof}
This is \cite[Lemma 4.18]{HPhighlytwisted}; see also \cite[Proposition~15]{GScanonicaltriangulation}.
\end{proof}

\begin{lemma}\label{Lem:maximal}
Suppose the volume functional on $\overline{\scrA(\calT_{K(p,q)})}$ is maximal at $\alpha$. Then $\alpha \notin \overline{\scrA(\calT_{K(p,q)})} \setminus \scrA(\calT_{K(p,q)})$.
\end{lemma}

\begin{proof}
We will prove the lemma by contradiction. Suppose $\alpha \in \overline{\scrA(\calT_{K(p,q)})} \setminus \scrA(\calT_{K(p,q)})$ maximises the volume functional. 

We first treat the case $p, q>2$. Notice that the simultaneous involutions $(a_1, b_1,c_1) \leftrightarrow (a_2, b_2,c_2)$ and $(a_4, b_4, c_4) \leftrightarrow (a_5, b_5,c_5)$ preserve Equations \eqref{eqn:w1} through \eqref{eqn:w12}. By concavity of the volume function, there is a maximiser such that $(a_1, b_1,c_1)  =(a_2, b_2,c_2)$ and $(a_4, b_4, c_4) = (a_5, b_5,c_5)$. Let  $(X, Y, Z)=(a_1, b_1,c_1)  =(a_2, b_2,c_2)$ and $(U, V, W)=(a_4, b_4, c_4) = (a_5, b_5,c_5)$, then Equations \eqref{eqn:w1} thru \eqref{eqn:w12} can be rewritten as follows (note Equations \eqref{eqn:w7} and \eqref{eqn:w8} become \eqref{eqn:n7}).
\begin{align}
2 \pi &= a_{0,1}+2b_0+b_3+2U+2X+2Y, \label{eqn:n1} \\
2 \pi &= a_{1,1}+b_3+2c_{0,1}+2 V, \label{eqn:n2}\\
2\pi &=  2W+2b_{0, 1} + \cdots + 2b_{p-3, 1}+ a_{p-3, 1}, \label{eqn:n5}\\
2\pi &=  2 c_{i+1, 1}+ a_{i+2, 1}+ a_{i, 1}  \quad \text{ for } 0 \leq i \leq p-5, \label{eqn:n8}\\
2\pi &= a_{p-4, 1}+2c_{p-3, 1} ,\label{eqn:n10}\\
2\pi &= W+Z+a_3+c_0, \label{eqn:n7}\\
2 \pi &= a_{0,2}+a_0+2c_3+2U+2V+2X, \label{eqn:n3}\\
2 \pi &= a_{1,2}+2 c_{0,2}+a_0+2Y, \label{eqn:n4}\\
2 \pi &=  2Z+2b_{0, 2} + \cdots + 2b_{q-3, 2}+ a_{q-3, 2},  \label{eqn:n6}\\
2\pi&= 2 c_{j+1, 2}+ a_{j+2, 2}+ a_{j, 2}  \quad \text{ for } 0 \leq j \leq q-5, \label{eqn:n9}\\
2\pi &= a_{q-4, 2}+2c_{q-3, 2}.\label{eqn:n11}
\end{align}

Take the sum of Equations \eqref{eqn:n1} and \eqref{eqn:n2} and inductively apply shape parameter equation $2c_{i, 1} = 2\pi -2b_{i,1}-2a_{i,1}$ and Equation \eqref{eqn:n8}, $2 \pi - a_{i, 1}=2c_{i+1 ,1}+a_{i+2, 1}$, for $0 \leq i \leq p-5$. Then apply shape parameter equation $2c_{p-4, 1} = 2\pi -2b_{p-4,1}-2a_{p-4,1}$ and Equation \eqref{eqn:n10} to obtain:
\begin{align}
4\pi &=  2c_{i+1, 1}+a_{i+1, 1}+a_{i+2,1}-2\sum_{j=0}^{i}b_{j,1}+2b_0+2b_3+2X+2Y+2U+2V  \nonumber \\
& \qquad \text{ for } 0 \leq i \leq p-5 \nonumber \\
&= 2c_{p-3, 1}+a_{p-3,1}-2\sum_{j=0}^{p-4}b_{j,1} +2b_0+ 2b_3 +2X+2Y+2U+2V. \label{eqn:14}
\end{align}
Next apply shape parameter equation $2c_{p-3, 1} = 2\pi -2b_{p-3,1}-2a_{p-3,1}$,  Equation \eqref{eqn:n5} and the shape parameter equations $2\pi=2W+2U+2V$ and $2X+2Y=2\pi-2Z$ to Equation \refeq{eqn:14}, to obtain
\begin{align}
Z &= b_0+b_3. \label{eqn:n12}
\end{align}

Similarly, the sum of Equations \eqref{eqn:n3} and \eqref{eqn:n4} can be rewritten as:
\begin{align}
W &= c_3+a_0\label{eqn:n13}
\end{align}

Let $V_i$ denote the $i$-th layered solid torus, $i=1$ or $2$. Throughout this proof we will apply Lemma~\ref{lemma:flatimpliestaut}: flat implies taut.

\vspace{.1in}

\noindent \underline{Case 1.} 
Suppose a tetrahedron in the layered solid torus $V_1$ is flat.  By Lemma~\ref{LSTflat}, all of the tetrahedra in  $V_1$ corresponding to $\alpha$ are taut. This means that the tetrahedron on its boundary is taut, and so exactly one of $a_{0,1}$, $b_{0,1}$, or $c_{0,1}$ is equal to $\pi$. We step through these cases to show they cannot maximise volume. 

\underline{Case 1(a).} Suppose $a_{0, 1} =\pi, b_{0, 1}=c_{0, 1}=0$. By Equation \eqref{eqn:n8}, $a_{2, 1} = \pi, c_{2, 1}=b_{2, 1} = 0$. By induction, $a_{i, 1}=\pi$, $c_{i, 1} =b_{i, 1} =0$ for $1 \leq i \leq p-3$. This is inconsistent with Equation \eqref{eqn:n10}. So we must have $a_{0, 1} =0$.

\underline{Case 1(b).} Suppose $b_{0, 1} =\pi, a_{0, 1}=c_{0, 1}=0$. By Equation \eqref{eqn:n5}, $b_{i, 1} =0$ for $1 \leq i \leq p-3$. 
By Equation \eqref{eqn:n8}, $a_{2, 1}=0$, hence $c_{2,1} =\pi$, and $c_{1,1}=\pi$. By induction $c_{i, 1} = \pi$ for $1 \leq i \leq p-3$. Furthermore, by Equation \eqref{eqn:n5}, $W=0$, and so one of $U$ or $V$ must equal $\pi$. By Equation \eqref{eqn:n13}, $c_3=a_0=0$. 

\begin{enumerate}
\item Suppose $U = \pi$. Then by Equation \eqref{eqn:n1}, $Y=X=b_3=b_0=0$ and so $Z=a_3=c_0=\pi$. By Equation \eqref{eqn:n6}, $b_{0,2}=\cdots =b_{p-2, 2}=0$. By Lemma \ref{LSTflat}, all tetrahedra in $V_2$ are also flat. Since all tetrahedra in $V_1$ are flat, and all tetrahedra $\Delta_0, \dots,\Delta_5$ are flat, this implies that the volume is zero, and so the volume functional cannot be maximised at $\alpha$.

\item Suppose $V = \pi$. By Equation \eqref{eqn:n3}, $a_{0,2}=X=0$.  Again by Lemma \ref{LSTflat}, all tetrahedra are flat, so the volume cannot be maximal.
\end{enumerate}

\underline{Case 1(c).} Suppose $c_{0, 1} =\pi, a_{0, 1}=b_{0, 1}=0$. By Equation \eqref{eqn:n8}, $c_{1, 1} = \pi, a_{2, 1}=0$ which implies that $a_{1,1}=0$.  By induction, $c_{i, 1} = \pi$  for $0 \leq i \leq p-4$ and by Equation \eqref{eqn:n10}, $c_{p-3, 1} = \pi$. By Equation \eqref{eqn:n5}, $W=\pi$. By Equation \eqref{eqn:n2}, $b_3=0$, and so one of $a_3=\pi$ or $c_3=\pi$. 

\begin{enumerate}
\item Suppose $a_3 = \pi$ and $c_3=0$. By Equation \eqref{eqn:n13}, $a_0=\pi$, and by Equation \eqref{eqn:n12}, $Z=0$. Since $c_3=U=V=0$, then by Equation \eqref{eqn:n3}, $a_{0,2}=\pi$.  By Lemma \ref{LSTflat}, all tetrahedra are flat, so the volume cannot be maximal.

\item  Suppose  $a_3 = 0$ and $c_3= \pi$. By Equation \eqref{eqn:n3} we have $a_{0,2}=a_0=X=0$. By Lemma \ref{LSTflat}, all tetrahedra are flat, so the volume cannot be maximal.
\end{enumerate}

\vspace{.1in}

\noindent\underline{Case 2.} Suppose a tetrahedron in the layered solid torus $V_2$ is flat.  By Lemma \ref{LSTflat}, all of the tetrahedra in $V_2$ corresponding to $\alpha$ are taut. In particular, exactly one of $a_{0,2}$, $b_{0,2}$, or $c_{0,2}$ is $\pi$ and the others are $0$. 

\underline{Case 2(a).}
Suppose $a_{0, 2} =\pi, b_{0, 2}=c_{0, 2}=0$. By Equation \eqref{eqn:n9}, $a_{2, 2} = \pi, c_{2, 2}=b_{2, 2} = 0$. By induction, $a_{i, 2}=\pi$, $c_{i, 2} =b_{i, 2} =0$ for $1 \leq i \leq q-3$. This is inconsistent with Equation \eqref{eqn:n11}. So we must have $a_{0, 2} =0$.

\underline{Case 2(b).} Suppose $b_{0, 2} =\pi, a_{0, 2}=c_{0, 2}=0$.  By Equation \eqref{eqn:n6}, $b_{i, 2} =0$ for $1 \leq i \leq q-3$ and $Z=0$, so either $X=\pi$ or $Y=\pi$. By Equation \eqref{eqn:n1}, $a_{0,1}=b_0=b_3=U=0$. By Lemma \ref{LSTflat}, all tetrahedra are flat, so the volume cannot be maximal.

\underline{Case 2(c).} Suppose $c_{0, 2} =\pi, a_{0, 2}=b_{0, 2}=0$. By Equation \eqref{eqn:n9}, $c_{1, 2} = \pi, a_{2, 2}=0$ which implies that $a_{1,2}=0$.  By induction, $c_{i, 2} = \pi$  for $0 \leq i \leq q-4$ and by Equation \eqref{eqn:n11}, $c_{q-3, 2} = \pi$. By Equation \eqref{eqn:n6}, $Z=\pi$ and by Equation \eqref{eqn:n4}, $a_0=0$. Then one of $b_0$ or $c_0$ equals $\pi$. 

\begin{enumerate}
\item Suppose $b_0=\pi$. Then 
  by Equation \eqref{eqn:n1}, $b_3=U=a_{0,1}=0$. By Lemma \ref{LSTflat}, all tetrahedra are flat.

\item Suppose $c_0=\pi$. Then by Equation \eqref{eqn:n12} $b_3=\pi$ and so by Equation \eqref{eqn:n13} $W=0$. This implies that $U=\pi$ or $V=\pi$. Since $b_3=\pi$, Equation \eqref{eqn:n2} implies that $V=0$ so $U=\pi$. By Equation \eqref{eqn:n1}, $a_{0,1}=0$.  By Lemma \ref{LSTflat}, all tetrahedra are flat.
\end{enumerate}


\vspace{.1in}

\noindent\underline{Case 3.} Now suppose no tetrahedra in $V_1$ and $V_2$ are flat, but one of the other tetrahedra is flat. There are a number of cases. We show that in each case, there is either a contradiction or a tetrahedron in $V_1$ or $V_2$ is flat. 
\begin{enumerate}
\item[(a)] If $b_0=\pi, U=\pi, X=\pi,$ or $Y=\pi$, then Equation \eqref{eqn:n1} implies $a_{0,1}=0$.
\item[(b)]  If $c_3=\pi$ or $V=\pi$,  then Equation \eqref{eqn:n3} implies $a_{0,2}=0$.
\item[(c)]  If $W = \pi$, then Equation \eqref{eqn:n5} implies $b_{0,1}=\cdots=b_{p-3, 1}=a_{p-3, 1}=0$.
\item[(d)] If $a_0 = \pi$, then Equation \eqref{eqn:n13} implies $W=\pi$, a contradiction by case (c).
\item[(e)]  If $Z = \pi$, then Equation \eqref{eqn:n6} implies $b_{0,2}=\cdots=b_{q-3, 2}=a_{q-3, 2}=0$.
\item[(f)] If $b_3=\pi$, then Equation \eqref{eqn:n12} implies $Z=\pi$, a contradiction by case (e).
\item[(g)] If $c_0 = \pi$, then Equation \eqref{eqn:n13} implies $c_3 = W$ and Equation \eqref{eqn:n3} implies $a_{0,2}=0$.
\item[(h)] If $a_3 = \pi$, then Equation \eqref{eqn:n12} implies $b_0 = Z$ and Equation \eqref{eqn:n1} implies $a_{0,1}=0$.
\end{enumerate}
This concludes the proof for $p,q>2$. 

Now consider $p=2$ and $q>2$. Equations \eqref{eqn:w7} thru \eqref{eqn:2q2} can be rewritten as follows. 
\begin{align}
2 \pi &= 2b_0+b_3+2U+2X+2Y \label{eqn:m1} \\
2 \pi &= b_3+2V+2W\label{eqn:m3}\\
2\pi &= W+Z+a_3+c_0 \label{eqn:m4}\\
2 \pi &= a_{0,2}+a_0+2c_3+2U+2V+2X \label{eqn:m5}\\
2\pi & = 2c_{0,2} + a_0 + a_{1,2} +2Y \label{eqn:m6old} \\
2 \pi &=  2Z+2b_{0, 2} + \cdots + 2b_{q-3, 2}+ a_{q-3, 2}  \label{eqn:m7}\\
2\pi&= 2 c_{j+1, 2}+ a_{j+2, 2}+ a_{j, 2}\text{ for } 0 \leq j \leq q-5 \label{eqn:m8}\\
2\pi &= a_{q-4, 2}+2c_{q-3, 2}\label{eqn:m9}
\end{align}
We may also apply the previous argument that obtained Equation \eqref{eqn:n13}, using Equations \eqref{eqn:m5} through \eqref{eqn:m9} and shape parameter equations, to obtain
\begin{equation} \label{eqn:m6}
W = c_3+a_0
\end{equation}

Suppose that one tetrahedron in the layered solid torus $V_2$ is taut. Then by following the previous argument $a_{0,2}=0$. 
\begin{enumerate}
\item If $b_{0,2}=\pi$, then Equation \eqref{eqn:m7} implies $Z=0$. If $X=\pi$ or $Y=\pi$, then Equation \eqref{eqn:m1} implies $b_0=b_3=U=0$. Therefore, all tetrahedra are flat. 

\item If $c_{0,2}=\pi$, then by following previous arguments $c_{i, 2}=\pi$ for $0\leq i \leq q-4$. Equation \eqref{eqn:m7} implies $Z=\pi$.
Equation \eqref{eqn:m6old} implies $a_0=0$. If $b_0=\pi$, then by Equation \eqref{eqn:m1}, all tetrahedra are flat. If $c_0=\pi$, then Equation \eqref{eqn:m4} implies $a_3=W=0$. Therefore, all tetrahedra are flat. 
\end{enumerate}

Now suppose no tetrahedra in the layered solid torus $V_1$ are flat, but some other tetrahedron is flat. Again we work through cases to obtain a contradiction. 
\begin{enumerate}
\item[(1)] If $c_3, U, V,$ or $X$ are equal to $\pi$ then by Equation \eqref{eqn:m5}, $a_{0,2}=0$. 
\item[(2)] If $Z=\pi$, then by Equation \eqref{eqn:m7}, $b_{0,2}=0$. 
\item[(3)] If $Y=\pi$, then by Equation \eqref{eqn:m6old}, $c_{0,2}=0$. 
\item[(4)] If $W=\pi$, then Equation \eqref{eqn:m3} implies $b_3=V=0$. By case~(1), $c_3\neq \pi$. If $a_3=\pi$, Equation \eqref{eqn:m4} implies $Z=c_0=0$, so $X$ or $Y$ is $\pi$, contradicting case (1) or (3). 
\item[(5)] If $b_0=\pi$, then $b_3=U=X=Y=0$ by Equation \eqref{eqn:m1}. So $Z=\pi$, contradicting case (2).
\item[(6)] If $a_0=\pi$, then Equation \eqref{eqn:m5} implies $c_3=0$ and $W=\pi$, contradicting case (4).
\item[(7)] If $c_0 = \pi$, then by Equations \eqref{eqn:m6}  we have $W=c_3$. Therefore, Equation \eqref{eqn:m5} and the shape parameter equation $U+V+W=\pi$ implies that  $ a_{0,2}=0$.
\item[(8)] If $a_3=\pi$, then Equation \eqref{eqn:m6} implies $W=a_0$. By Equation \eqref{eqn:m4}, $Z=b_0$ and so by Equation \eqref{eqn:m1}, $U=0$. But this implies that $V=\pi$ or $W=\pi$ which is a contradiction of case (1) or (4).
\item[(9)] If $b_3=\pi$, then Equation \eqref{eqn:m6} implies $W=a_0$. By Equation \eqref{eqn:m4}, $\pi+b_0 = Z$ which implies that $Z=\pi$, a contradiction by case (2). 
\end{enumerate}

For $K(p,2)$ where $p>2$, the argument is the same as that for $K(2,q)$. More carefully, switch the pairs of tetrahedra $\Delta_0$ and $\Delta_3$, $\Delta_1$ and $\Delta_5$, $\Delta_2$ and $\Delta_4$, and replace $q$ with $p$, taking the layered solid torus $V_1$ rather than $V_2$. Then up to relabelling, the identical argument as above proves the result for $K(p,2)$. 

For $p=q=2$, rewrite Equations \eqref{eqn:w7}, \eqref{eqn:w8}, \eqref{eqn:2q1}, \eqref{eqn:2q2}, \eqref{eqn:p21}, and \eqref{eqn:p22} as follows:
\begin{align}
2\pi &= W+Z+a_3+c_0 \label{eqn:k221}\\
2 \pi &= 2b_0+b_3+2U+2X+2Y\label{eqn:k222} \\
2 \pi &= b_3+2V+2W\label{eqn:k223}\\
2 \pi &=   a_0 +2X+ 2c_3 +2U+2V \label{eqn:k224}\\
2 \pi&= a_0 +2Y+2Z\label{eqn:k225}
\end{align}

\begin{enumerate}
\item[(1)]  If $X, Y, U, V, c_3$ or $b_0=\pi$, then by Equation \eqref{eqn:k222} and \eqref{eqn:k224} all tetrahedra are flat.
\item[(2)]  If $Z=\pi$, then Equation \eqref{eqn:k225} implies $a_0=0$ so by case (1) we must have $c_0=\pi$. By Equation \eqref{eqn:k221}, $a_3=W=0$ so all  tetrahedra are flat.
\item[(3)]  If $W=\pi$, then Equation \eqref{eqn:k223} implies $b_3=0$ so by case (1) we must have $a_3=\pi$. By Equation \eqref{eqn:k221}, $Z=c_0=0$ so all  tetrahedra are flat. 
\item[(4)]  If $a_3=\pi$, then Equation \eqref{eqn:k223} implies $U=0$ so either $V=\pi$ or  $W=\pi$. By case (1) or (3), all  tetrahedra are flat.
\item[(5)]  If $c_0=\pi$ then Equation \eqref{eqn:k225} implies that $Y=\pi$ or $Z=\pi$. So by case (1) or (2), all  tetrahedra are flat.
\item[(6)] If $b_3=\pi$, then Equation \eqref{eqn:k223} implies $2U=\pi$. Equation \eqref{eqn:k222} implies $Z=\pi$. By case (2), $W=0$, a contradiction to $2U=\pi$ by Lemma \ref{lemma:flatimpliestaut}. 
\item[(7)] If $a_0=\pi$, then Equation \eqref{eqn:k225} implies $2X=\pi$. Equation \eqref{eqn:k224} implies $W=\pi$. By case (3), $Z=0$, a contradiction to $2X=\pi$ by Lemma \ref{lemma:flatimpliestaut}. \qedhere
\end{enumerate}
\end{proof}

\begin{proof}[Proof of Theorem~\ref{Thm:geometric}]
Suppose  $p,q \geq 2$. By Lemma \ref{Lem:nonempty} the space of angle structures $\scrA(\calT_{K(p,q)})$ of the ideal triangulation $\calT_{K(p,q)}$ is non-empty. By Lemma \ref{Lem:maximal}, the volume functional on the closure of this space is maximal in the interior. It follows from Casson and Rivin's theorem~\cite[Theorem~1.2]{FuterGueritaud:Angled} that $\calT_{K(p,q)}$ is a geometric triangulation of the complement of $K(p,q)$. 
\end{proof}

\section{Triangulations and Neumann--Zagier equations}\label{Sec:NeumannZagier}

In this section we show that for any $p$, $q$, the deformation variety of $K(p,q)$ is cut out by eight equations in nine variables. Eliminating all but two variables from these equations gives (a factor of) the $\PSL$ A--polynomial. The tools we use are from \cite{HMP}. We begin by reviewing this material, following the notation of Neumann and Zagier~\cite{NeumannZagier}.

\subsection{Background on Ptolemy equations}
Recall labelled triangulations of Definition~\ref{Def:LabelledTriang}. 
Consider the edge $(v w)$ of a tetrahedron $\Delta$ where $v, w \in \{0,1,2,3\}$. We will denote by $j(v w)$ the index of the edge to which  $(v w)$ is glued. In particular, $(v w)$ is glued to $E_{j(v w)}$.

\begin{definition}
The \emph{$a$-incidence number} ($b$-, $c$-incidence number, respectively) of $E_k$ with the tetrahedron $\Delta_j$, denoted by $a_{k,j}$ ( $b_{k,j}, c_{k,j}$, respectively), is the number of $a$-edges ($b$-, $c$-edges, respectively) of $\Delta_j$ identified to $E_k$.

Let $m_k$ and $l_k$ denote the meridian and longitude of the $k^{th}$ boundary tori. Then the \emph{$a$-incidence number} ($b$-, $c$-incidence number, respectively) of $m_k$ ($l_k$ respectively) with the tetrahedron $\Delta_j$ is the signed number of segments of $m_k$ ($l_k$ respectively) running through a corner of a triangle corresponding to an $a$-edge ($b$-, $c$-edge, respectively) of $\Delta_j$, where the sign is $+1$ if the segment cuts off the corner in the anticlockwise direction, and $-1$ if it cuts off the corner in the clockwise direction. 
The $a$-incidence number ($b$-, $c$-incidence number, respectively) of $m_k$ with the tetrahedron $\Delta_j$ is denoted by $a^m_{k,j}$  ( $b_{k,j}^m, c_{k,j}^m$, respectively). The $a$-incidence number  ($b$-, $c$-incidence number, respectively)  of $l_k$ with the tetrahedron $\Delta_j$ is denoted by $a^l_{k,j}$  ( $b_{k,j}^l, c_{k,j}^l$, respectively).
\end{definition}

\begin{definition}
Let $\mathcal{T}$ be a labelled triangulation of $M$. Assign variables $z_j, z_j',$ and $z_j''$ to the $a$-, $b$-, and $c$- edges of $\Delta_j$, respectively such that  
\begin{equation}\label{eqn:shapeparameters}
z_j z_j' z_j'' = -1 \quad \mbox{ and } \quad z_j + (z_j')^{-1}-1=0
\end{equation}
Equation \eqref{eqn:shapeparameters} are called \emph{shape parameter equations}. If $\Delta_j$ has a hyperbolic structure then the shape parameter equations are standard terahedron parameters as discussed in \cite{thurstonsnotes}. 
\end{definition}

We now consider a link complement $M$ with triangulation $\mathcal{T}$ and choose generators $m_i$ and $l_i$ of  $H_1(\mathbb{T}_i)$ for $1 \leq i \leq {\frak{n}}_{\frak{c}}$ where $\mathbb{T}_i$ is a boundary torus of $\overline{M}$ and ${\frak{n}}_{\frak{c}}$ is the total number of torus boundary components of $\overline{M}$. 

Suppose an ideal triangulation of $M$ has $\frak{n}$ edges, and hence $\frak{n}$ tetrahedra.
For each edge $E_k$, $k=1, \dots, \frak{n}$, the \emph{gluing equation} for edge $E_k$ is
\begin{equation}\label{eq:gluingequation}
\prod_{j=1}^{\frak{n}} z_j^{a_{k,j}} (z_j')^{b_{k,j}}  (z_j'')^{c_{k,j}} =1
\end{equation}
and the \emph{cusp equations} for cusp torus $\mathbb{T}_k$ are
\begin{equation}\label{eqn:cusp}
m_k = \prod_{j=1}^{\frak{n}} z_j^{a_{k,j}^m} (z_j')^{b_{k,j}^m}  (z_j'')^{c_{k,j}^m} \quad \mbox{ and } \quad
l_k = \prod_{j=1}^{\frak{n}} z_j^{a_{k,j}^l} (z_j')^{b_{k,j}^l}  (z_j'')^{c_{k,j}^l}
\end{equation}

Simultaneously solving the shape parameter equations, gluing equations, and cusp equations with $m_k, l_k=1$ gives a complete hyperbolic structure on the given link complement. When there is just one cusp, solving these equations in terms of $m$ and $l$ gives a factor of the $\PSL(2, \CC)$ $A$-polymonial~\cite{ChampApoly}. 

We will change this system of equation by a change of basis, as in~\cite{HMP}. This will replace gluing and cusp equations, which potentially have very high degree in the variables $z_i$, $z_i'$ and $z_i''$, with Ptolemy equations, which have degree at most two in new variables $\lambda_i$, $\lambda_i'$. To do so, we recall information on the Neumann--Zagier data. 

The \emph{incidence matrix} of $M$ with the given triangulation and generators, denoted by $\In$, is the $(\frak{n}+2\frak{n}_{\frak{c}}) \times 3\frak{n}$ matrix whose $k^{th}$ row for $1 \leq k \leq \frak{n}$ is given by
\begin{equation}\label{eqn:incidencerow}
\begin{pmatrix} a_{k,1} & b_{k,1} & c_{k,1}, \dots, a_{k,\frak{n}} & b_{k,\frak{n}} & c_{k,\frak{n}}\end{pmatrix}.
\end{equation}
The next rows are given in pairs:
\begin{equation}\label{eqn:incidencerowml}
\begin{pmatrix} a_{k,1}^m & b_{k,1}^m & c_{k,1}^m, \dots, a_{k,\frak{n}}^m & b_{k,\frak{n}}^m & c_{k,\frak{n}}^m \\ a_{k,1}^l & b_{k,1}^l & c_{k,1}^l, \dots, a_{k,\frak{n}}^l & b_{k,\frak{n}}^l & c_{k,\frak{n}}^l \end{pmatrix}
\end{equation}

The \emph{Neumann--Zagier matrix}, denoted by $\NZ$, is the $(\frak{n}+2\frak{n}_{\frak{c}}) \times 2\frak{n}$ matrix whose first $\frak{n}$ rows 
are given by
\begin{equation}
R_k^G = \begin{pmatrix} a_{k,1}-c_{k,1} & b_{k,1}-c_{k,1}  & \cdots & a_{k,\frak{n}}-c_{k,\frak{n}} & b_{k,\frak{n}}-c_{k,\frak{n}} \end{pmatrix}
\end{equation}
and whose last $2\frak{n}_c$ rows are given by pairs
\begin{equation}
\binom{R^{m}_k}{R^l_k} =
\begin{pmatrix}
  a_{k,1}^m-c_{k,1}^m & b_{k,1}^m-c_{k,1}^m & \cdots & a_{k,\frak{n}}^m-c_{k,\frak{n}}^m & b_{k,\frak{n}}^m-c_{k,\frak{n}}^m  \\
  a_{k,1}^l-c_{k,1}^l & b_{k,1}^l-c_{k,1}^l & \cdots & a_{k,\frak{n}}^l-c_{k,\frak{n}}^l & b_{k,\frak{n}}^l-c_{k,\frak{n}}^l
\end{pmatrix}
\end{equation}

Let $c_k$, $c_i^m$, $c_i^l$ denote the sum of all the c-incidence numbers in a row. That  is,
\[ c_k = \sum_{j=1}^{\frak{n}} c_{k,j}, \quad c_i^m = \sum_{j=1}^{\frak{n}} c_{i,j}^m, \quad c_i^l = \sum_{j=1}^{\frak{n}} c_{i,j}^l, \quad \mbox{for $1 \leq k \leq \frak{n}$ and $1 \leq i \leq \frak{n}_{\frak{c}}$}.
\] 
Define a vector called the \emph{$C$-vector} by 
\begin{equation}
\vec{C}  := \begin{pmatrix}
2-c_1 & \dots & 2-c_{\frak{n}} & -c_1^m & -c_1^l & \dots &  -c_{\frak{n}_{\frak{c}}}^m & -c_{\frak{n}_{\frak{c}}}^l 
\end{pmatrix},
\end{equation}

Neumann showed in \cite[Theorem~2.4]{NeumannNZmatrix} that there exists an integer vector $\vec{B}$ satisfying $\NZ\cdot \vec{B}=\vec{C}$. We call such a vector a \emph{$B$-vector}. 

All of this is equivalent to the Neumann-Zagier datum in \cite{DGNZdatum}. 

We use this to change the variables in the gluing and cusp equations, as in the following theorem. 

\begin{theorem}[Theorem~1.1 of \cite{HMP}]\label{Thm:PtolemyEquations}
Let $M$ be a knot complement with a hyperbolic triangulation $\mathcal{T}$, with associated Neumann--Zagier matrix $\NZ$ and $C$-vector $\vec{C}$. Define formal variables $\gamma_1, \dots, \gamma_n$, one associated with each edge of $\mathcal{T}$. For a tetrahedron $\Delta_j$ of $\mathcal{T}$, and edge $\alpha \beta \in \{ 01, 02, 03, 12, 13, 23\}$, define $\gamma_{j(\alpha \beta)}$ to be the variable $\gamma_k$ such that the edge of $\Delta_j$ between vertices $\alpha$ and $\beta$ is glued to the edge of $\mathcal{T}$ associated with $\gamma_k$. 

For each tetrahedron $\Delta_j$ of $\mathcal{T}$, define \emph{Ptolemy equations} of $\mathcal{T}$ in terms of $\Delta_j$ by 
\[
(-1)^{B_j'} l^{-\mu_j/2} m^{\lambda_j/2} \gamma_{j(01)} \gamma_{j(23)}+(-1)^{B_j} l^{-\mu_j'/2} m^{\lambda_j'/2} \gamma_{j(02)} \gamma_{j(13)}-\gamma_{j(03)} \gamma_{j(12)}=0.
\]

After setting $\gamma_i=$ for some $1 \leq i \leq n$ and solving the system of Ptolemy equations of $\mathcal{T}$ in terms of $m$ and $l$ , we obtain a factor of the $\mathit{PSL}(2, \mathbb{C})$ $A$-polynomial of $K$, $A_K(M,L)$ where $m = M^2$ and $l^{1/2}=-L$.
\end{theorem}

Note that Theorem~\ref{Thm:PtolemyEquations} differs slightly from \cite{HMP}: There, rather than using the full Neumann--Zagier matrix, a reduced matrix is used, with linearly dependent rows removed. However, note that the entries of such a row do not play any role in the conclusion of the theorem. The $B$-vector has the same entries, since the extra row(s) give redundant gluing equations.

\subsection{A-polynomials of double twist knots}
We use the (conjecturally) minimal triangulation of Section~\ref{Sec:mintriangulation} to compute the defining equations for the A-polynomial of double twist knots $K(p,q)$.

Start by first considering the knot $K(2,2)=7_4$ ($p=2$). The Ptolemy equations of this knot will differ from the rest of the family because the cusp neighbourhood is constructed without the layered solid tori. In particular, it is constructed from Figure~\ref{fig:cuspntwoprismsml} by removing tetrahedra $\Delta_6$, $\Delta_7$, $\Delta_8$, $\Delta_9$ and then folding edges ${\mathcal{O}_0^1}$ and  ${\mathcal{O}_0^2}$. 
Its Neumann--Zagier matrix is given in Table~\ref{Tab:NZMatrix74knot}.

\begin{table}[ht]
\centering
 \begin{tabular}{ c||c|c||c|c||c|c||c|c||c|c||c|c}
     \multicolumn{1}{c}{}&  \multicolumn{2}{c}{ $\Delta_{0}$}  &  \multicolumn{2}{c}{ $\Delta_{1}$}  & \multicolumn{2}{c}{ $\Delta_{2}$}&  \multicolumn{2}{c}{ $\Delta_{3}$}  & \multicolumn{2}{c}{ $\Delta_{4}$}   & \multicolumn{2}{c}{ $\Delta_{5}$}  
     \\ \hline \hline
 $E_0$ &  $-1$ &   $-1$ &  &  & $-1$ & $-1$  &  $1$ &  &  &  & $-1$ &   $-1$ 
\\ \hline 
 $E_1$ &  $-1$  &  $-1$ & $-1$ &  $-1$   &  &  &  $1$  & & $-1$&  $-1$ & &  
\\ \hline 
 ${\calP^1}$   &  &  &  &  &  &  & &  $1$ & $-1$ &    & $-1$ &  
\\ \hline
 ${\mathcal{O}_0^1}$ & &  $2$ &  $1$  &  $1$   &   $1$ &  $1$   &    &  $1$ &  $1$  &   &  $1$ & 
\\ \hline
 ${\calP^2}$  &  $1$ & &  $-1$  &    & &    $-1$ &  &   &  &    &  &   
\\ \hline
 ${\mathcal{O}_0^2}$   &  $1$ &    &   $1$ &   &  $1$ &    &  $-2$ &  $-2$   &  $1$ &  $1$ &   $1$ &  $1$    
\\ \hline
  $m$   &  &  $1$   &    &  & $1$   &  $1$ &  &   $1$  &   &  &  &
\\ \hline
  $l$ &  $2$  &  &  $2$ &   &  $-2$ &  &  $-2$  &   $-2$ &  &  $2$ &  $2$   &  
\\ \hline
\end{tabular}
\caption{Neumann--Zagier matrix of $K(2,2)$.}
\label{Tab:NZMatrix74knot}
\end{table}

The $C$-vector is given by 
\[
\vec{C} = \begin{pmatrix}
-1 &
-1 &
0 &
2 &
0 &
0 &
1 &
0
\end{pmatrix}^T
\]
and the $B$-vector below satisfies $\NZ\cdot \vec{B} = \vec{C}$. 
\[
\vec{B} = \begin{pmatrix}
1 &
0 &
1 &
-1 &
0 &
0 &
0 &
1 &
1 &
-1 &
0 &
0
\end{pmatrix}^T
\]

Then we obtain the following six Ptolemy equations, one for each of the tetrahedra. 
\begin{align*}
\Delta_0 &: m \gamma_{\mathcal{P}^2}\gamma_{\mathcal{O}_0^2} - l^{-1/2} \gamma_{\mathcal{O}_0^1}^2 - \gamma_{E_0}\gamma_{E_1}  \\
\Delta_1 &: -m \gamma_{\mathcal{O}_0^1} \gamma_{\mathcal{O}_0^2} - \gamma_{\mathcal{O}_0^1} \gamma_{\mathcal{P}^2} - \gamma_{E_1} \gamma_{\mathcal{P}^2} \\
\Delta_2 &:  l^{-1/2} m^{-1} \gamma_{\mathcal{O}_0^1} \gamma_{\mathcal{O}_0^2} +  l^{-1/2} \gamma_{\mathcal{O}_0^1} \gamma_{\mathcal{P}^2}-\gamma_{E_0}\gamma_{\mathcal{P}^2} \\
\Delta_3 &:   -m^{-1} \gamma_{E_0} \gamma_{E_1} +  l^{-1/2} m^{-1} \gamma_{\mathcal{O}_0^1} \gamma_{\mathcal{P}^1}-\gamma_{\mathcal{O}_0^2}^2 \\
\Delta_4 &: -   \gamma_{\mathcal{O}_0^1} \gamma_{\mathcal{O}_0^2} -   m^{1} \gamma_{\mathcal{O}_0^2} \gamma_{\mathcal{P}^1}-\gamma_{E_1}\gamma_{\mathcal{P}^1} \\
\Delta_5 &:   m^{1} \gamma_{\mathcal{O}_0^1} \gamma_{\mathcal{O}_0^2} +   \gamma_{\mathcal{P}^1} \gamma_{\mathcal{O}_0^2}-\gamma_{E_0}\gamma_{\mathcal{P}^1} 
\end{align*}

By letting $m = M^2$, $l^{1/2} = -L$, $\gamma_{\calP^2}=1$, and by taking the Gr\"oebner basis of the above Ptolemy equations, we obtain the $\PSL$-$A$-polynomial of $K(2,2)$:
\begin{align*}
   A_{K(2,2)}(L,M) &= (-1 + M) (1 + M) (1 + M^2) (-1 + L M^4)  \\
   & \times (1 - L + L M^2 + 2 L M^4 + L M^6 - L M^8 + L^2 M^8) \\
  & \times (1 - 2 L + L^2 + 6 L M^2 - 2 L^2 M^2 + 2 L M^4 + 3 L^2 M^4 - 7 L M^6 \\
&+ 2 L^2 M^6 + 2 L M^8 - 7 L^2 M^8 + 3 L M^{10} + 2 L^2 M^{10} - 2 L M^{12} + 6 L^2 M^{12} \\
&+ L M^{14} - 2 L^2 M^{14} + L^3 M^{14}).
\end{align*}
The last two factors agree with the $A$-polynomial computed using $SL_2(\mathbb{C})$ representations of $\pi_1(S^3 - K(2,2))$.


For at least one of $p,q \geq 3$, we will compute the $A$-polynomial using results from~\cite{HMP}, using Dehn filling. 

Let $M$ be the complement of the Borromean rings, $M(p,q)$ the complement of $K(p,q)$. As detailed in Section~\ref{Sec:mintriangulation}, we obtain a triangulation of $M(p,q)$ from the 10-tetrahedra triangulation of $M$ in Section~\ref{Sec:10TetB} by first removing $\{ \Delta_6, \Delta_7, \Delta_8, \Delta_9\}$, to obtain a triangulated space $M_0^1$, and then adding in tetrahedra, one tetrahedron per step through the Farey graph in each of two layered solid tori.

Let $M_1^1$ be obtained from $M_0^1$ by attaching the first tetrahedron $\Delta_1^1$ of the first layered solid torus to $M_0^1$. Inductively, let $M_j^1$ be obtained by attaching the $j$-th tetrahedron $\Delta_j^1$ of the first layered solid torus to $M_{j-1}^1$, for $j=1, \dots, p-2$. Let $M_1^2$ be obtained from $M_{p-2}^1$ by attaching the first tetrahedron $\Delta_1^2$ to $M_{p-2}^1$. Inductively, let $M_k^2$ be obtained by attaching the $k$-th tetrahedron $\Delta_k^2$ of the second layered solid torus to $M_{k-1}^2$. 

Let $\NZ$ be the Neumann--Zagier matrix of the Borromean rings complement $M$,  excluding the rows corresponding to the meridian and longitude of each crossing circle cusp, the edge incident to only $\{  \Delta_6, \Delta_7 \}$, and the edge incident to only $\{  \Delta_8, \Delta_9 \}$. Also, let $C$ be the corresponding $C$-vector.

\begin{proposition}\label{Prop:NZBandCvector} \
\begin{enumerate}
\item The Neumann--Zagier matrix $\text{NZ}_0^1$ of $M_0^1$ is obtained from $\NZ$ by deleting the columns corresponding to the removal of tetrahedra
$\{\Delta_6, \Delta_7,  \Delta_8, \Delta_9\}$. 

\item Inductively, the Neumann--Zagier matrix $\NZ_j^1$ of $M_j^1$ is obtained from $\NZ_{j-1}^1$ by adding a pair of columns associated to the tetrahedron $\Delta_j^1$ and a row associated to the edge $\mathcal{O}^1_{j+2}$ for $0 \leq j \leq p-4$. The matrix $\NZ_{p-3}^1$ is obtained from $\NZ_{p-4}^1$ by adding a pair of columns associated to the tetrahedron $\Delta_{p-3}^1$. The entries of the new rows outside of the new columns are zero and the non-zero entries of each column is given in Table~\ref{NZmatrixfilledconstruction2}. For the second layered solid torus, the Neumann--Zagier matrices $\NZ_j^2$ are obtained from $\NZ_{j-1}^2$ in a similar manner. The final matrix is $\NZ_{q-3}^2$. 

\item The $C$-vector associated to $\NZ_{q-3}^2$, denoted by $C(p,q)$, is obtained from $C$ by adding $2$ to the entries corresponding to $\mathcal{O}_0^1$ and $\mathcal{O}_0^2$, subtracting $2$ from the entries corresponding to $\mathcal{O}_1^1$ and $\mathcal{O}_1^2$, and appending $0$'s associated to edges $\mathcal{O}_j^1$ for $2 \leq j \leq p-2$ and $\mathcal{O}_k^2$ for $2 \leq k \leq q-2$. 
\end{enumerate}
\end{proposition}
\begin{table}[H]
  \begin{subtable}{.4\textwidth}
    \centering
    \begin{tabular}{ l||c|c}
      \multicolumn{1}{l}{}&  \multicolumn{2}{c}{ $\Delta_{j}^i$}  \\ \hline\hline
      $\mathcal{P}^i$ & & $2$  \\ \hline
      $\mathcal{O}_{j}^i$ & $1$ &   \\ \hline
      $\mathcal{O}_{j+1}^i$ & $-2$  & $-2$ \\ \hline
      $\mathcal{O}_{j+2}^i$ &  $1$&   
    \end{tabular}
    \caption{$i=1$ and $0 \leq j \leq p-4$  or $i=2$ and $0\leq j\leq q-4$}
    \label{table:newcolumns}
  \end{subtable}
  \begin{subtable}{.4\textwidth}
    \centering
    \begin{tabular}{ l||c|c}
      \multicolumn{1}{l}{}&  \multicolumn{2}{c}{ $\Delta_{j}^i$}   \\ \hline \hline
      $\mathcal{P}^i$ &  $1$ & $2$   \\ \hline
      $\mathcal{O}_{j}^i$ & $1$    \\ \hline
      $\mathcal{O}_{j+1}^i$  & $-2$  & $-2$ 
    \end{tabular}
    \caption{$i=1$ and $j=p-3$ or $i=2$ and $j = q-3$}
  \end{subtable}
  \caption{Non-zero entries in the columns associated to $\Delta_{\mathcal{O}_j^i}$.}
  \label{NZmatrixfilledconstruction2}
\end{table}

\begin{proof}
The change to $\NZ$ when the triangulation is given by Dehn filling using a layered solid torus follows from \cite[Proposition~3.11]{HMP}. In our case, we start with the 10-tetrahedron triangulation of the Borromean rings complement, and Dehn fill using the layered solid torus as described in Section~\ref{Sec:mintriangulation}. 
\end{proof}

\begin{lemma}\label{Lem:CandBvectors}
The $C$- and $B$- vectors for $S^3-K(p,p)$ obtained from performing Dehn filling on the cusps of the two crossing circles from are

\begin{align}
  C &= \begin{pmatrix}
    -1 &
    -1 &
    0 &
    2 &
    0 &
    0 &
    0 &
    \cdots &
    0 &
    1 &
    0
  \end{pmatrix}^T \text{ and }  \label{eqn:Cvect} \\
  B &= \begin{pmatrix}
    0 &
    1 &
    0 &
    0 &
    0 &
    0 &
    0 &
    0 &
    0 &
    0 &
    \cdots&
    0
  \end{pmatrix}^T \label{eqn:Bvect}
\end{align}
\end{lemma}

\begin{proof}
The Neumann--Zagier matrix $\NZ$ associated to the triangulated Borromean rings complement of Section~\ref{Sec:10TetB} is given in Table~\ref{NZbeforedehncontruction2}.

{\centering
  \resizebox{\columnwidth}{!}{%
    \renewcommand*{\arraystretch}{1}
    \begin{tabular}{ c||c|c||c|c||c|c||c|c||c|c||c|c|||c|c||c|c||c|c||c|c|}
      \multicolumn{1}{c}{}&  \multicolumn{2}{c}{\fontsize{6}{6}$\Delta_{0}$}  &  \multicolumn{2}{c}{\fontsize{6}{6}$\Delta_{1}$}  & \multicolumn{2}{c}{\fontsize{6}{6}$\Delta_{2}$}&  \multicolumn{2}{c}{\fontsize{6}{6}$\Delta_{3}$}  & \multicolumn{2}{c}{\fontsize{6}{6}$\Delta_{4}$}   & \multicolumn{2}{c}{\fontsize{6}{6}$\Delta_{5}$}   & \multicolumn{2}{c}{\fontsize{6}{6}$\Delta_{6}$} &
      \multicolumn{2}{c}{\fontsize{6}{6}$\Delta_{7}$} &
      \multicolumn{2}{c}{\fontsize{6}{6}$\Delta_{8}$} & \multicolumn{2}{c}{\fontsize{6}{6}$\Delta_{9}$}\\ \hline \hline
      \fontsize{6}{6}$E_0$ & \fontsize{6}{6}$-1$ & \fontsize{6}{6}$-1$  &  &  &\fontsize{6}{6}$-1$ &\fontsize{6}{6}$-1$  & \fontsize{6}{6}$1$ &  &  &  &\fontsize{6}{6}$-1$ &  \fontsize{6}{6}$-1$&  &  &  &  &  &  &  & \\ \hline 
      \fontsize{6}{6}$E_1$ & \fontsize{6}{6}$-1$ &\fontsize{6}{6}$-1$ &\fontsize{6}{6}$-1$ & \fontsize{6}{6}$-1$   &  &  & \fontsize{6}{6}$1$  & &\fontsize{6}{6}$-1$& \fontsize{6}{6}$-1$ & &  &  &  &  &  &  &  & & \\ \hline 
      \fontsize{6}{6}$\calP^1$   &  &  &  &  &  &  & &  &\fontsize{6}{6}$-1$ &  \fontsize{6}{6}$-1$  &\fontsize{6}{6}$-1$ &  \fontsize{6}{6}$-1$ &   & &  &  &  \fontsize{6}{6}$1$ & & \fontsize{6}{6}$1$  &   \\ \hline
      \fontsize{6}{6}$\calO_0^1$ &  & \fontsize{6}{6}$2$ & \fontsize{6}{6}$1$  & \fontsize{6}{6}$1$   &  \fontsize{6}{6}$1$ & \fontsize{6}{6}$1$   &    & \fontsize{6}{6}$1$ & \fontsize{6}{6}$1$  &   & \fontsize{6}{6}$1$ &  &  &    &  &   &  \fontsize{6}{6}$-1$ & \fontsize{6}{6}$-1$ & \fontsize{6}{6}$-1$  & \fontsize{6}{6}$-1$  \\ \hline
      \fontsize{6}{6}$\calO_1^1$ &   &  &  &  &  &  &  & \fontsize{6}{6}$1$  &  &  \fontsize{6}{6}$1$   &    & \fontsize{6}{6}$1$ &  &  &  &  &   & \fontsize{6}{6}$1$ &  & \fontsize{6}{6}$1$ \\ \hline 
      \fontsize{6}{6}$\calP^2$  & & & \fontsize{6}{6}$-1$  &  \fontsize{6}{6}$-1$  &\fontsize{6}{6}$-1$  &   \fontsize{6}{6}$-1$ &  &   &  &    &  &   & \fontsize{6}{6}$1$ &   & \fontsize{6}{6}$1$ &   &  &  &  &  \\ \hline
      \fontsize{6}{6}$\calO_0^2$   & \fontsize{6}{6}$1$  &   &  \fontsize{6}{6}$1$ &   & \fontsize{6}{6}$1$ &    & \fontsize{6}{6}$-2$ & \fontsize{6}{6}$-2$   & \fontsize{6}{6}$1$ & \fontsize{6}{6}$1$ &  \fontsize{6}{6}$1$ & \fontsize{6}{6}$1$     &  \fontsize{6}{6}$-1$ & \fontsize{6}{6}$-1$ &  \fontsize{6}{6}$-1$&  \fontsize{6}{6}$-1$  &   &   &  &  \\ \hline
      \fontsize{6}{6}$\calO_1^2$  & \fontsize{6}{6}$1$ &  &  &  \fontsize{6}{6}$1$ &    & \fontsize{6}{6}$1$ &    &  & &  &  &  &  &  \fontsize{6}{6}$1$ &   & \fontsize{6}{6}$1$ &   &  &    &  \\ \hline
      \fontsize{6}{6} $m$   &  & \fontsize{6}{6}$1$  &    &  &\fontsize{6}{6}$1$   & \fontsize{6}{6}$1$ &  &  \fontsize{6}{6}$1$  &   &  &  &   &  &  &  &    &   &   &  &  \\
      \hline
      \fontsize{6}{6} $l$ & \fontsize{6}{6}$2$ &   & \fontsize{6}{6}$2$ &   & \fontsize{6}{6}$-2$ &  & \fontsize{6}{6}$-2$  &  \fontsize{6}{6}$-2$ &  & \fontsize{6}{6}$2$ & \fontsize{6}{6}$2$   &  &  &  &  &  &  &  &   &    \\
      \hline
\end{tabular}}}
\captionof{table}{Neumann-Zagier matrix NZ before Dehn filling.} \label{NZbeforedehncontruction2}

The corresponding $C$-vector is
$C' = \begin{pmatrix} -1 &  -1 &  0 &  0 &  2 &  0 &  -2 &  2 &  1 &  0 \end{pmatrix}^T $. A $B$-vector $B'$ satisfying $\NZ\cdot B'=C'$ is given by
\begin{equation}\label{Bvector1}
  B' = \begin{pmatrix}
    2 &
    -1 &
    0 &
    0 &
    0 &
    0 &
    0 &
    2 &
    0 &
    0 &
    0 &
    0 &
    0 &
    0 & 
    0 &
    0 &
    0 &
    0 &
    0 &
    0
  \end{pmatrix}^T
\end{equation}

By Proposition~\ref{Prop:NZBandCvector}, the $C$-vector associated to $\NZ_{q-3}^2$ is that of Equation~\eqref{eqn:Cvect}. 

We have that $\NZ_{q-3}^2$ for $p\geq 3$ is obtained by removing the columns associated to tetrahedra $\{\Delta_6, \Delta_7, \Delta_8 \Delta_9\}$ then adding rows and columns associated to tetrahedra $\Delta_j^i$ where the non-zero entries of the added rows and columns are given in Table~\ref{NZmatrixfilledconstruction2}. Therefore, column 2 of $NZ(r)$ is $C$. It follows that a $B$-vector is given by Equation~\eqref{eqn:Bvect}. 
\end{proof}

\begin{theorem}
For $p, q\geq 3$, the following 
equations are defining equations for the $\PSL$-$A$-polynomial of the double twist knot $K(p,q)$, where $0 \leq j < p-3$ and $0 \leq i < q-3$.
\begin{align} 
  \Delta_0 \: &:\: - m \gamma_{\mathcal{O}_0^2} \gamma_{\mathcal{O}_1^2} +  l^{-1/2}  \gamma_{\mathcal{O}_0^1}^2 -\gamma_{E_0}\gamma_{E_1} \label{Ptolemyoutside1}\\
  \Delta_1 \: & :\: m \gamma_{\mathcal{O}_0^1} \gamma_{\mathcal{O}_0^2} + \gamma_{\mathcal{O}_0^1} \gamma_{\mathcal{O}_1^2}-\gamma_{E_1}\gamma_{\calP^2} \\
  \Delta_2 \: & :\:  l^{-1/2} m^{-1} \gamma_{\mathcal{O}_0^1} \gamma_{\mathcal{O}_0^2} +  l^{-1/2} \gamma_{\mathcal{O}_0^1} \gamma_{\mathcal{O}_1^2}-\gamma_{E_0}\gamma_{\calP^2} \\
  \Delta_3 \: & :\:    m^{-1} \gamma_{E_0} \gamma_{E_1} +  l^{-1/2} m^{-1} \gamma_{\mathcal{O}_0^1} \gamma_{\mathcal{O}_1^1}-\gamma_{\mathcal{O}_0^2}^2 
\end{align}
\begin{align}
  \Delta_4 \: & :\:    \gamma_{\mathcal{O}_0^1} \gamma_{\mathcal{O}_0^2} +  m \gamma_{\mathcal{O}_1^1} \gamma_{\mathcal{O}_0^2}-\gamma_{E_1}\gamma_{\calP^1} \\
  \Delta_5 \: & :\:   m \gamma_{\mathcal{O}_0^1} \gamma_{\mathcal{O}_0^2} +  \gamma_{\mathcal{O}_1^1} \gamma_{\mathcal{O}_0^2}-\gamma_{E_0}\gamma_{\calP^1} \label{Ptolemyoutside6}\\
  \Delta_j^1 \: & :\: \gamma_{\mathcal{O}_j^1}\gamma_{\mathcal{O}_{j+2}^1} + \gamma_{\mathcal{P}^1}^2 - \gamma_{\mathcal{O}_{j+1}^1}^2 \quad j=0, \dots, p-4 \label{Ptolemyfilled1}\\
  \Delta_i^2 \: & :\: \gamma_{\mathcal{O}_i^2}\gamma_{\mathcal{O}_{i+2}^2} + \gamma_{\mathcal{P}^2}^2 - \gamma_{\mathcal{O}_{i+1}^2}^2 \quad i=0, \dots, q-4 \\
  \Delta_{p-3}^1 \: & :\:   \gamma_{\calP^1} \gamma_{\mathcal{O}_{p-3}^1} +  \gamma_{\calP^1}^2-\gamma_{\mathcal{O}_{p-2}^1}^2 \\
  \Delta_{q-3}^2 \: & :\:   \gamma_{\calP^2} \gamma_{\mathcal{O}_{q-3}^2} + \gamma_{\calP^2}^2-\gamma_{\mathcal{O}_{q-2}^2}^2 \label{Ptolemyfilled4}
\end{align}
\end{theorem}

\begin{proof}
The Ptolemy equations corresponding to tetrahedra $\Delta_0, \dots, \Delta_5$, outside the layered solid tori, are obtained from Theorem~\ref{Thm:PtolemyEquations} using the $B$-vector given in Lemma~\ref{Lem:CandBvectors} and the first twelve columns in Table~\ref{NZbeforedehncontruction2}.  The Ptolemy equations corresponding to tetrahedra $\Delta_j^i$ in the layered solid tori are obtained from Theorem~\ref{Thm:PtolemyEquations} using the $B$-vector given in Lemma~\ref{Lem:CandBvectors} and the adjusted Neumann--Zagier matrix $\NZ_{q-3}^2$ obtained by applying Proposition~\ref{Prop:NZBandCvector} to Table~\ref{NZmatrixfilledconstruction2}.
\end{proof}

\begin{example}
For $p=3$ and $q=3$, the defining equations for the $A$-polynomial are the six Ptolemy equations given in Equations~\eqref{Ptolemyoutside1} through \eqref{Ptolemyoutside6} and two Ptolemy equations
\begin{eqnarray*} 
  \Delta_0^1 \: & :\:   \gamma_{\calP^1} \gamma_{\mathcal{O}_{0}^1} +  \gamma_{\calP^1}^2-\gamma_{\mathcal{O}_{1}^1}^2 \\
  \Delta_0^2 \: & :\:   \gamma_{\calP^2} \gamma_{\mathcal{O}_{0}^2} + \gamma_{\calP^2}^2-\gamma_{\mathcal{O}_{1}^2}^2
\end{eqnarray*}

By letting $m=M^2, l^{1/2} = -L$, $\gamma_{\mathcal{O}_1^2}  =1$ and by taking the Groebner basis of the defining equations, we obtain the following equation.

\begin{align*}
  A_{K(3,3)}(M, L) &= (1 + L M^2)^3 \\
  &(1 - 4 L + 6 L^2 - 4 L^3 + L^4 + 12 L M^2 - 
  26 L^2 M^2 + 16 L^3 M^2 - 2 L^4 M^2 \\
  & \quad - 2 L M^4 + 26 L^2 M^4 -   18 L^3 M^4 + 3 L^4 M^4 - 3 L M^6 + 36 L^2 M^6 - L^3 M^6 \\
  & \quad  -   4 L^4 M^6 + 8 L M^8 - 31 L^2 M^8 + 32 L^3 M^8 + 5 L^4 M^8 - 
  13 L M^{10} \\
  & \quad - 20 L^2 M^{10} - 6 L^3 M^{10} + 4 L^4 M^{10} + 4 L M^{12} - 
  6 L^2 M^{12} - 20 L^3 M^{12} \\
  & \quad - 13 L^4 M^{12} + 5 L M^{14} + 32 L^2 M^{14} - 
  31 L^3 M^{14} + 8 L^4 M^{14} - 4 L M^{16} \\
  & \quad - L^2 M^{16} + 36 L^3 M^{16} - 
  3 L^4 M^{16} + 3 L M^{18} - 18 L^2 M^{18} + 26 L^3 M^{18} \\
  & \quad - 2 L^4 M^{18} - 
  2 L M^{20} + 16 L^2 M^{20} - 26 L^3 M^{20} + 12 L^4 M^{20} + L M^{22} \\
  & \quad  - 
  4 L^2 M^{22} + 6 L^3 M^{22} - 4 L^4 M^{22} + L^5 M^{22}) 
\end{align*}
\begin{align*}
  & (1 - 4 L + 
  6 L^2 - 4 L^3 + L^4 + 6 L M^2 - 15 L^2 M^2 + 12 L^3 M^2 - 
  3 L^4 M^2 \\
  & \quad + 6 L M^4 - 7 L^2 M^4 + L^4 M^4 - 3 L M^6 + 25 L^2 M^6 - 
  26 L^3 M^6 + 5 L^4 M^6 \\
  & \quad + 11 L^2 M^8 - 5 L^4 M^8 + 3 L M^{10} - 
  6 L^2 M^{10} + 24 L^3 M^{10} - 6 L^4 M^{10} \\
  & \quad  - 2 L M^{12} + 8 L^2 M^{12} + 
  8 L^3 M^{12} + 8 L^4 M^{12} - 2 L^5 M^{12} - 6 L^2 M^{14} \\
  & \quad + 24 L^3 M^{14} -  6 L^4 M^{14} + 3 L^5 M^{14} - 5 L^2 M^{16} + 11 L^4 M^{16} + 5 L^2 M^{18}\\
  & \quad  -   26 L^3 M^{18} + 25 L^4 M^{18} - 3 L^5 M^{18} + L^2 M^{20} - 7 L^4 M^{20} \\
  &  \quad +    6 L^5 M^{20} - 3 L^2 M^{22} + 12 L^3 M^{22} - 15 L^4 M^{22} + 6 L^5 M^{22} + 
  L^2 M^{24} \\
  &  \quad  - 4 L^3 M^{24} + 6 L^4 M^{24} - 4 L^5 M^{24} + L^6 M^{24}).
\end{align*}

The last two factors agree with the $A$-polynomial computed using $SL_2(\mathbb{C})$ representations of $\pi_1(S^3 - K_{3,3})$.
\end{example}

\begin{theorem}[Theorem~3.11 of \cite{ThompsonApoly}]\label{Thm:layeredsolidtoruspoly}
Suppose a knot is obtained from a link complement by Dehn filling using a layered solid torus. Suppose the tail of the layered solid torus has length $n \geq 1$ and begins at step $k$, and suppose the folding equation corresponds to the tip being in the same direction as the tail. Consider the tetrahedron with slopes $(o_k, f_k, p_k, h_k)$ obtained from taking a walk in the Farey graph from triangle $(o_k, f_k, p_k)$ to triangle $(f_k, p_k, h_k)$ where $p_k$ is the pivot slope, that is, $p_k = p_{k+1} = \cdots = p_{k+n-1}$. The folding equation, along with the set of tail equations, is equivalent to the equation
\begin{equation}
  H_n - \gamma_{f_k}^{n-1} \gamma_{o_k}^n\gamma_{p_k} =0,
\end{equation}
where 
\begin{equation}
  H_n = \gamma_{f_k}^{2n} + \sum_{a+b \leq n-1} (-1)^{n-a-b}\binom{n-1-a}{b} \binom{n-b}{a} \gamma_{f_k}^{2a} \gamma_{o_k}^{2b}  \gamma_{p_k}^{2(n-a-b)}.
\end{equation}
\end{theorem}

\begin{corollary}\label{Cor:Apoly}
For $p, q\geq 3$, the following eight equations are defining equations for the $\PSL$-$A$-polynomial of the double twist knot $K(p,q)$.
\begin{align} 
  \Delta_0 \: & :\: - m \gamma_{\mathcal{O}_0^2} \gamma_{\mathcal{O}_1^2} +  l^{-1/2}  \gamma_{\mathcal{O}_0^1}^2 -\gamma_{E_0}\gamma_{E_1} \label{Ptolemyoutside1a}\\
  \Delta_1 \: & :\: m \gamma_{\mathcal{O}_0^1} \gamma_{\mathcal{O}_0^2} + \gamma_{\mathcal{O}_0^1} \gamma_{\mathcal{O}_1^2}-\gamma_{E_1}\gamma_{\calP^2} \\
  \Delta_2 \: & :\:  l^{-1/2} m^{-1} \gamma_{\mathcal{O}_0^1} \gamma_{\mathcal{O}_0^2} +  l^{-1/2} \gamma_{\mathcal{O}_0^1} \gamma_{\mathcal{O}_1^2}-\gamma_{E_0}\gamma_{\calP^2} \\
  \Delta_3 \: & :\:    m^{-1} \gamma_{E_0} \gamma_{E_1} +  l^{-1/2} m^{-1} \gamma_{\mathcal{O}_0^1} \gamma_{\mathcal{O}_1^1}-\gamma_{\mathcal{O}_0^2}^2 \\
  \Delta_4 \: & :\:    \gamma_{\mathcal{O}_0^1} \gamma_{\mathcal{O}_0^2} +  m \gamma_{\mathcal{O}_1^1} \gamma_{\mathcal{O}_0^2}-\gamma_{E_1}\gamma_{\calP^1} \\
  \Delta_5 \: & :\:   m \gamma_{\mathcal{O}_0^1} \gamma_{\mathcal{O}_0^2} +  \gamma_{\mathcal{O}_1^1} \gamma_{\mathcal{O}_0^2}-\gamma_{E_0}\gamma_{\calP^1} \label{Ptolemyoutside6a} \\
  \text{Tail 1} \: & :\: H_{p-2}^1 - \gamma_{\mathcal{O}_1^1}^{p-3} \gamma_{\mathcal{O}_0^1}^{p-2}\gamma_{\calP^1} \\
  \text{Tail 2} \: & :\: H_{q-2}^2 - \gamma_{\mathcal{O}_1^2}^{q-3} \gamma_{\mathcal{O}_0^2}^{q-2}\gamma_{\calP^2}
\end{align}
where,
\begin{align}
  H_n^1 &= \gamma_{\mathcal{O}_1^1}^{2n} + \sum_{a+b \leq n-1} (-1)^{n-a-b}\binom{n-1-a}{b} \binom{n-b}{a} \gamma_{\mathcal{O}_1^1}^{2a} \gamma_{\mathcal{O}_0^1}^{2b}  \gamma_{\calP^1}^{2(n-a-b)}\\
  H_m^2 &=  \gamma_{\mathcal{O}_1^2}^{2m} + \sum_{a+b \leq m-1} (-1)^{m-a-b}\binom{m-1-a}{b} \binom{m-b}{a} \gamma_{\mathcal{O}_1^2}^{2a} \gamma_{\mathcal{O}_0^2}^{2b}  \gamma_{\calP^2}^{2(m-a-b)}.
\end{align}
\end{corollary}

\bibliographystyle{amsplain}
\bibliography{bib.bib}

\end{document}